\documentclass[11pt,a4paper]{amsart}
\usepackage[latin1]{inputenc}
\usepackage{amsmath,amssymb,amsthm}
\usepackage{graphicx}
\usepackage[left=2.5cm,right=2.5cm,top=2.5cm,bottom=2.5cm]{geometry}
\usepackage{color}

\def\exp{{\mathrm{exp}}}

\newcommand{\Z}{\mathbb{Z}}

\newcommand{\R}{\mathbb{R}}
\newcommand{\rr}{\mathbb{R}}

\newcommand{\dx}{\Delta x}
\newcommand{\dt}{\Delta t}
\newcommand{\dxi}{\Delta\xi}
\newcommand{\ds}{\Delta s}

\newcommand{\sign}{\mathop{\mathrm{sign}}}
\newcommand{\eps}{\varepsilon}

\newtheorem{defin}{Definition}[section]
\newtheorem{theorem}{Theorem}[section]
\newtheorem{prop}{Proposition}[section]

\newtheorem{lemma}{Lemma}[section]
\newtheorem{remark}{Remark}

\numberwithin{equation}{section}

\def\umu{u^\mu}

\title[Large-time asymptotics, vanishing viscosity and numerics]{Large-time asymptotics, vanishing viscosity \\ and numerics for 1-D scalar conservation laws}
\author[L. I. Ignat, A. Pozo, E. Zuazua]{L. I. Ignat, A. Pozo, E. Zuazua}

\address{L. I. Ignat
\hfill\break\indent Institute of Mathematics ``Simion Stoilow'' of the Romanian Academy\\
\hfill\break\indent 21 Calea Grivitei Street \\010702 Bucharest \\ Romania
\hfill\break\indent \and
\hfill\break\indent Faculty of Mathematics and Computer Science, University of Bucharest\\
\hfill\break\indent 14 Academiei Street \\010014 Bucharest \\ Romania
\hfill\break\indent \and
\hfill\break\indent BCAM - Basque Center for Applied Mathematics\\
 \hfill\break\indent Alameda de Mazarredo 14, 48009 Bilbao, Basque Country - Spain.}
 \email{{\tt liviu.ignat@gmail.com}\hfill\break\indent {\it Web page: }{\tt http://sites.google.com/site/liviuignat}}
\address{A. Pozo
\hfill\break\indent BCAM - Basque Center for Applied Mathematics\\
\hfill\break\indent Alameda de Mazarredo 14, E-48009 Bilbao, Basque Country - Spain.}
 \email{{\tt pozo@bcamath.org}\hfill\break\indent {\it Web page: }{\tt http://www.bcamath.org/pozo}}
\address{E.Zuazua
\hfill\break\indent BCAM - Basque Center for Applied Mathematics\\
\hfill\break\indent Alameda de Mazarredo 14, 48009 Bilbao, Basque Country - Spain \\
\hfill\break\indent \and
\hfill\break\indent Ikerbasque, Basque Foundation for Science \\
 \hfill\break\indent Alameda Urquijo 36-5, Plaza Bizkaia, 48011 Bilbao, Basque Country - Spain.}
 \email{{\tt zuazua@bcamath.org}\hfill\break\indent {\it Web page: }{\tt http://www.bcamath.org/zuazua}}

\keywords{asymptotic behavior, conservation laws, monotone conservative schemes, N-waves}
\subjclass[2010]{35B40,35L65,65M12}

\begin{document}
\begin{abstract}
	In this paper we analyze the large time asymptotic behavior of the discrete solutions of numerical approximation schemes for scalar hyperbolic conservation laws. We consider three monotone conservative schemes that are consistent with the one-sided Lipschitz condition (OSLC): Lax-Friedrichs, Engquist-Osher and Godunov. We mainly focus on the inviscid Burgers equation, for which we know that the large time behavior is of self-similar nature, described by a two-parameter family of N-waves. We prove that, at the numerical level, the large time dynamics depends on the amount of numerical viscosity introduced by the scheme: while Engquist-Osher and Godunov yield the same N-wave asymptotic behavior, the Lax-Friedrichs scheme leads to viscous self-similar profiles, corresponding to the asymptotic behavior of the solutions of the continuous viscous Burgers equation. The same problem is analyzed in the context of self-similar variables that lead to a better numerical performance but to the same dichotomy on the asymptotic behavior: N-waves versus viscous ones. We also give some hints to extend the results to more general fluxes. Some numerical experiments illustrating the accuracy of the results of the paper are also presented.		
\end{abstract}

\maketitle

\section{Introduction and main results}

This paper is devoted to the analysis of the asymptotic behavior as $t \to \infty$ for $1-D$ scalar hyperbolic conservation laws of the form
\begin{equation}\label{eq:hyp}
	u_t+\big[f(u)\big]_x=0,\quad\quad x\in\R,t>0.
\end{equation}
We shall mainly focus on the Burgers equation with a quadratic flux $f(u)=u^2/2$:
\begin{equation}\label{eq:hyp.burgers}
	u_t+\left(\frac{u^2}{2}\right)_x=0,\quad\quad x\in\R,t>0.
\end{equation}

The asymptotic behavior of the solutions of the hyperbolic Burgers equation is well known to be of self-similar nature (see \cite{Liu1984}). Indeed, as $t \to \infty$ the solutions develop a $N$-wave behavior, conserving the mass of the initial datum that is invariant under the evolution. Note however that the mass does not suffice to identify the asymptotic self-similar profile that belongs to a two-parameter family of solutions, these parameters corresponding to two invariants of the system: the positive and the negative masses. In particular, generically, the N-wave corresponding to solutions emanating from changing sign initial data changes sign.

The asymptotic behavior differs significantly for the viscous version of these models:
\begin{equation}\label{eq:par}
	u^\eps_t+\big[f(u^\eps)\big]_x=\eps u^\eps_{xx},\quad\quad x\in\R,t>0,
\end{equation}
and
\begin{equation}\label{eq:par.burgers}
	u^\eps_t+\left[\frac{(u^\eps)^2}{2}\right]_x=\eps u^\eps_{xx},\quad\quad x\in\R,t>0.
\end{equation}
Indeed, for $\varepsilon>0$ these problems are of parabolic nature and, as $t$ tends to infinity, the solutions behave in a self-similar way with a viscous profile of constant sign that is fully determined by the conserved mass (see \cite{Hopf1950}).

Of course, for finite time, the solutions of the viscous models \eqref{eq:par} and \eqref{eq:par.burgers} are well known to converge to the entropy solutions of the hyperbolic scalar conservation laws \eqref{eq:hyp} and \eqref{eq:hyp.burgers}, respectively; but, as shown above, this limit can not be made uniform as time tends to infinity. Indeed, roughly, we could say that the vanishing viscosity and large time limits do not commute and that, accordingly, the following two limits yield to different results:
\begin{equation}
	\lim_{t\to\infty}\lim_{\eps\to0} u^\eps(x,t) \qquad \mbox{ and } \qquad \lim_{\eps\to0}\lim_{t\to\infty} u^\eps(x,t).
\end{equation}

While the first limit leads to the two-parameter hyperbolic N-waves, possibly changing sign, the second one leads to a more restrictive class of asymptotic profiles, corresponding to the N-waves of constant sign. This issue has been precisely analyzed for the quadratic nonlinearity $f(u) = u^2/2$ (see, for instance, \cite{Kim2002,Kim2001} and the references therein). In particular, in \cite{Kim2001} the authors describe the transition from the N-wave shape ---the asymptotic profile of the inviscid equation--- to the diffusion wave ---the asymptotic profile in the viscous equation---.

The main result of this paper states that the same can occur when approximating the hyperbolic equations \eqref{eq:hyp} and \eqref{eq:hyp.burgers} by numerical schemes. This is not so surprising since, as it is well known, convergent numerical schemes introduce some degree of numerical viscosity. Our analysis allows classifying numerical schemes in those that, as time tends to infinity, introduce a negligible amount of numerical viscosity, and therefore lead to the correct asymptotic behavior described by the N-waves, and those that introduce too much numerical viscosity thus leading to viscous self-similar profiles. As we shall see, Engquist-Osher and Godunov schemes belong to the first category while the classical Lax-Friedrichs scheme to the second one. Summarizing, we can say that the solutions of the Engquist-Osher and Godunov schemes, for a fixed mesh, capture the hyperbolic dynamics of the continuous systems; the Lax-Friedrichs scheme, because of the excess of numerical viscosity, leads to the wrong asymptotic behavior, of viscous nature and not of hyperbolic one.

Our results, corresponding to the $L^1$-setting, exhibit a significant difference with respect to previous works regarding conservative monotone schemes. In \cite{Harabetian:1988}, the author analyses the large-time behavior of these schemes in the context of rarefaction waves, thus rather corresponding to a $L^\infty$-setting. Our case can be formally understood as the limit one in which both values at $\pm \infty$ vanish and, hence, reveals the second term in the asymptotic expansion of solutions. We show that, in this framework, the extra viscosity added by the schemes has to be handled carefully to detect the asymptotic behavior as time tends to infinity of discrete solutions.

This issue is important in applications where solutions need to be computed for long time intervals. This occurs for instance, in the context of the design of supersonic aircrafts where sonic-boom minimization is one of the key issues (see \cite{Alonso:2012}). Note that, although analysis is limited to hyperbolic models, the same conclusions are also to be taken into account when numerically approximating viscous conservation laws where the amount of asymptotic effective viscosity as $t$ tends to infinity may very significantly depend on the nature of the numerical scheme under consideration.

The main goal of the present paper is to analyze the asymptotic behavior as $n\to\infty$ of these discrete solutions for $\dx$ and $\dt$ fixed. Of course, we are interested on numerical schemes that are well known to converge to the entropy solution of \eqref{eq:hyp} and with mesh-size parameters satisfying the corresponding CFL condition. Let us now introduce more precisely the numerical schemes under consideration. Given some grid size $\dx$ and time step $\dt$, we consider $u_j^n$ to be the approximation of $u(n\dt,j\dx)$, obtained by a conservative numerical scheme that approximates equation \eqref{eq:hyp} (e.g. Chapter III in \cite{0768.35059}),
\begin{equation}\label{num}
	u^{n+1}_j=u_n^j-\frac{\dt}{\dx}\left(g^n_{j+1/2}-g^n_{j-1/2}\right),\quad\quad j\in\Z, n>0,
\end{equation}
where $g_{j+1/2}^n=g(u_{j},u_{j+1})$ is the numerical flux, an approximation of the continuous flux $f$ by a continuous function $g:\R^{2}\to\R$.

Our analysis is mainly concerned with the numerical schemes of Lax-Friedrichs, Engquist-Osher and Godunov. They are of conservative nature, and well-known to converge to the entropy solution of \eqref{eq:hyp} under suitable CFL conditions and satisfy the so-called one-sided Lipschitz condition (OSLC) that is required to establish, in particular, decay properties as the discrete time tends to infinity. To be more precise, let us consider $\{u^0_j\}_{j\in \Z}$ an approximation of the initial data; for instance
\begin{equation}\label{discr.inidata}
	u^0_j=\frac 1{\Delta x}\int _{x_{j-1/2}}^{x_{j+1/2}}u_0(x)dx,\quad x_{j+1/2}=(j+\frac12)\Delta x.
\end{equation}
We introduce the piecewise constant function $u_\Delta$ defined almost everywhere in $[0,\infty)\times \R$ by
\begin{equation}\label{udelta}
	u_\Delta(t,x)=u_j^n, \quad x_{j-1/2}<x<x_{j+1/2}, \, t_n\leq t< t_{n+1},
\end{equation}
where $t_n=n\Delta t$ and $u_j^n$ is computed by \eqref{num}. Here and subsequently, for $v=\{v_j\}_{j\in\Z}$ and $p\in[1,\infty)$, we use the following discrete norms:
\begin{equation*}
	\|v\|_{p,\Delta} = \Big(\Delta x \sum_{j\in\Z} |v_j|^p\Big)^{1/p}, \qquad	\|v\|_{\infty,\Delta} = \max_{j\in\Z} |v_j| , \qquad	TV(v) = \sum_{j\in\Z} |v_{j+1}-v_j|.
\end{equation*}

The following theorem, focused on the Burgers equation, is the main result of this paper.
\begin{theorem}\label{asymptotic}
Let $u_0\in L^1(\R)$ and choose mesh-size parameters $\dx$ and $\dt$ satisfying the CFL condition $\lambda\|u^n\|_{\infty,\Delta}\le1$, $\lambda=\dt/\dx$. Let $u_\Delta$ be the corresponding solution of the discrete scheme \eqref{num} for the hyperbolic Burgers conservation law \eqref{eq:hyp.burgers}. Then, for any $p\in [1,\infty)$, the following holds 
\begin{equation}\label{limit.infty}
	\lim _{t\rightarrow \infty} t^{\frac 12(1-\frac 1p)}\|u_\Delta(t)-w(t)\|_{L^p(\rr)}=0,
\end{equation}
where the profile $w$ is as follows:
\begin{enumerate}
	\item for the Lax-Friedrichs scheme, $w=w_{M_\Delta}$ is the unique solution of the continuous viscous Burgers equation
		\begin{equation}\label{limit1}
			\begin{cases}
				w_t+\Big(\frac{w^2}{2}\Big)_x=\frac{(\Delta x)^2}{2\Delta t}w_{xx},&x\in\R,t>0,\\[10pt]
				w(0)=M_\Delta\delta_0,
			\end{cases}
		\end{equation}
with $M_\Delta=\int_{\rr}u^0_\Delta$.
\vspace{0.5cm}
	\item for Engquist-Osher and Godunov schemes, $w=w_{p_\Delta,q_\Delta}$ is the unique solution of the hyperbolic Burgers equation
		\begin{equation}\label{limit2}
			\begin{cases}
				w_t+\Big(\frac{w^2}{2}\Big)_x=0,\quad x\in\R,t>0,\\
				w(0)=M_\Delta\delta_0,	\quad			{\displaystyle \lim _{t \rightarrow 0}\int _{-\infty}^x w(t,z)dz=}
				\begin{cases} 0, &x<0,\\ -p_\Delta, &x=0,\\ q_\Delta-p_\Delta, &x>0, \end{cases}
			\end{cases}
		\end{equation}
with $M_\Delta=\int_{\rr}u^0_\Delta$ and
\begin{equation*}
	p_\Delta=-\min_{x\in \rr}\int _{-\infty}^x u^0_\Delta(z)dz \quad \mbox{ and } \quad q_\Delta=\max_{x\in \rr}\int _x^{\infty}u^0_\Delta(z)dz.
\end{equation*}
\end{enumerate}
\end{theorem}

The initial data in the above equations \eqref{limit1} and \eqref{limit2} have to be understood in the sense of the convergence of bounded measures. We refer to \cite{MR1233647} and \cite{Liu1984} for a precise definition.

It is well known \cite{0762.35011,0902.35002} that the above profiles are explicitly given by 
\begin{equation}\label{wm}
	w_{M_\Delta}(x,t)=-\frac{2\sqrt{\nu}}{ t^{1/2} } \exp\left(-\frac{x^2}{4\nu t}\right)\left[C_{M_\Delta}+ \displaystyle\int_{-\infty}^{x/\sqrt{\nu t}} \exp\left(-\frac{s^2}{4}\right)ds\right]^{-1},
\end{equation}
where $\nu=\Delta x^2/(2\Delta t)$ and $C_{M_\Delta}$ is such that the mass of the solution $w_{M_\Delta}$ is ${M_\Delta}$, and
\begin{equation}\label{nwave}
	w_{p_\Delta,q_\Delta}(x,t)=\begin{cases} \frac{x}{t}, & -\sqrt{2p_\Delta t}<x<\sqrt{2q_\Delta t},\\
	 0, & \mbox{elsewhere}.
	 \end{cases}
\end{equation}

Note that the viscous profiles \eqref{wm} are fully determined by the total mass, which is conserved under the dynamics under consideration both in the time-continuous and time-discrete case. By the contrary, the N-wave profiles \eqref{nwave} are uniquely determined by the two parameters $(p,q)$ of invariants that are constant along the continuous and discrete dynamics. The quantity $q_\Delta-p_\Delta$ is precisely $M_\Delta$, the mass of function $u^0_\Delta$.

The difference among them can be observed in Figure \ref{fig:examples}, for instance, where we have taken $\dx=1/100$, $\dt=1/100$, $M_\Delta=1/10$, $p_\Delta=1/10$ and $q_\Delta=1/5$.

\begin{figure}[t!]
	\includegraphics[width=0.8\linewidth]{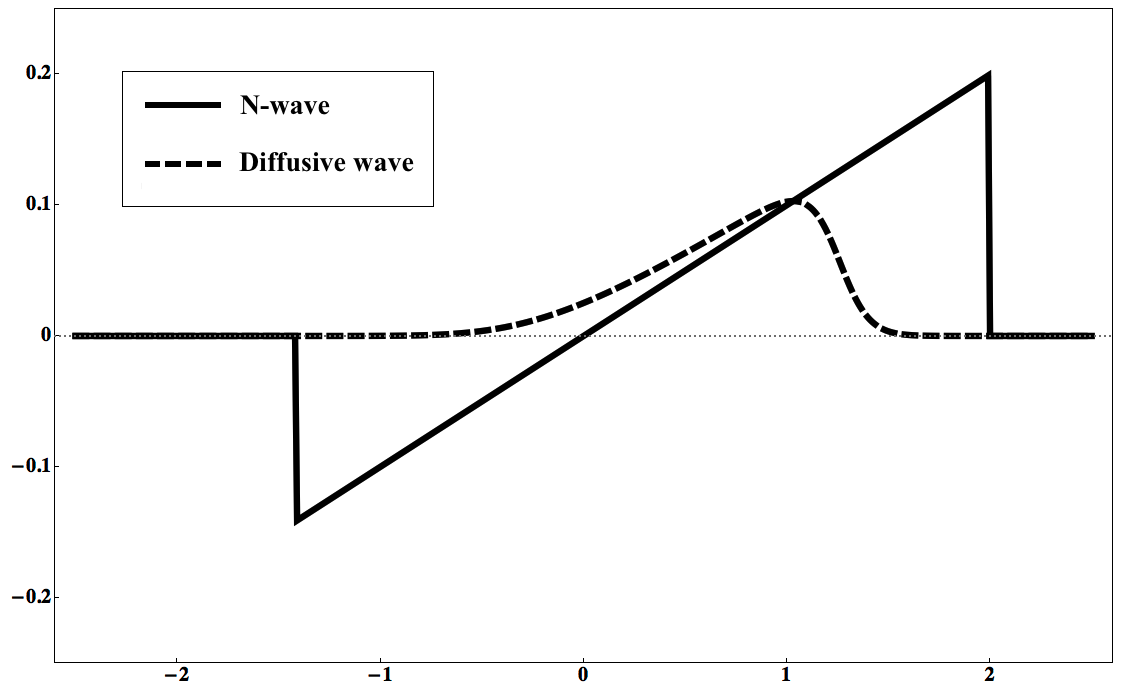}
	\caption{Diffusive wave and N-wave evaluated at $t=10$, with $\dx=1/100$, $\dt=1/100$, $M_\Delta=1/10$, $p_\Delta=1/10$ and $q_\Delta=1/5$.}
	\label{fig:examples}
\end{figure}

\begin{remark}
It is important to emphasize that, with this theorem, we classify the most classical numerical schemes into two different types. Although the grid parameters $\dx$ and $\dt$ are fixed, as the discrete time tends to infinity, solutions develop a continuous in time behavior that, depending on the numerical scheme, can be of hyperbolic or parabolic nature. On the other hand, we remark that the asymptotic profile of the discrete solutions depends on the way we choose the approximation of initial data $\{u^0_j\}_{j\in \Z}$ in \eqref{num}, since the values of $p_\Delta$ and $q_\Delta$ are connected to that discretization. Nevertheless, $u^0_\Delta$ converges to the continuous initial data and, accordingly, the same happens to $p_\Delta$, $q_\Delta$ and $M_\Delta$.
\end{remark}

In order to prove Theorem \ref{asymptotic}, we use scaling arguments, similar to those applied in the proofs of the continuous analogues. Besides, we also introduce similarity variables, which are also a standard tool in the analysis of asymptotic behavior of partial differential equations. This will allow us to observe some phenomena in a more clear manner.

The rest of this paper is divided as follows: in Section \ref{section2} we present some classical facts about the numerical approximation of one-dimensional conservation laws and obtain preliminary results that will be used in the proof of the main results of this paper. In Section \ref{section3} we prove the main result, Theorem \ref{asymptotic}, and we illustrate it in Section \ref{section4} with numerical simulations. In Section \ref{section5}, we analyze the same issues in the similarity variables and compare the results to the approximations obtained directly from the physical ones. Finally, in Section \ref{section6} we discuss possible generalizations to other numerical schemes and to more general fluxes (uniformly convex or odd ones).


\section{Preliminaries}\label{section2}
Following \cite{MR923922} and \cite{0768.35059}, we recall some well-known results about numerical schemes for 1D scalar conservation laws. We prove some new technical results that will be used in Section \ref{section3} in the proof of Theorem \ref{asymptotic}. We restrict our attention to the Burgers equation \eqref{eq:hyp.burgers}. More general results will be discussed in Section \ref{section5} for uniformly convex  and odd fluxes.

First, given a time-step $\dt$ and a uniform spatial grid $\Delta$ with space increment $\dx$, we approximate the conservation law
\begin{equation}\label{eq:conslaw}
	\begin{cases}
		u_t+\left(\frac{u^2}{2}\right)_x=0,&x\in\R, t>0,\\
		u(x,0)=u_0(x),&x\in\R,
	\end{cases}
\end{equation}
by an explicit difference scheme of the form:
\begin{equation}\label{eq:scheme}
	u_j^{n+1}=H(u_{j-k}^n,\dots,u_{j+k}^n), \quad\forall n\ge0,j\in\Z,
\end{equation}
where $H:\R^{2k+1}\to\R$, $k\geq 1$, is a continuous function and $u_j^n$ denotes the approximation of the exact solution $u$ at the node $(n\dt,j\dx)$. We assume that there exists a continuous function $g:\R^{2k}\to\R$, called numerical flux, such that
\begin{equation*}
	H(u_{-k},\dots,u_{k})=u_0-\lambda\left[g(u_{-k+1},\dots,u_{k})-g(u_{-k},\dots,u_{k-1})\right],\quad \lambda=\Delta t/\Delta x,
\end{equation*}	
so that scheme \eqref{eq:scheme} can be put in conservation form. This means that setting $g_{j+1/2}^n=g(u_{j-k+1},\dots,u_{j+k})$, we can rewrite scheme \eqref{eq:scheme} as:
\begin{equation}\label{eq:consscheme}
	u_j^{n+1}=u_j^n-\lambda\left(g_{j+1/2}^n-g_{j-1/2}^n\right), \quad\forall n\ge0,j\in\Z,
\end{equation}
where $\{u^0_j\}_{j\in \Z}$ is an approximation of $u_0\in L^1(\rr)$, defined, for instance, as in \eqref{discr.inidata}.

 It is obvious that if a scheme can be put in conservation form, then the mass of the solution is conserved in time.

We will focus our analysis on monotone schemes. We recall that a numerical scheme \eqref{eq:scheme} is said to be monotone if function $H$ is monotone increasing in each of its arguments.

Let us remark that any 3-point ($k=1$) monotone scheme in conservation form satisfies that their numerical flux $g(u,v)$ is an increasing function in the first argument and decreasing in the second one. The consistency of the scheme also guarantees that
\begin{equation}
	g(u,u)=\frac{u^2}{2},\quad \forall u\in\R.
\end{equation}

Now, we recall a classical result about conservative schemes. For the sake of simplicity, we denote $H_\Delta(v)=\{H(v_{j-k},\dots,v_{j+k})\}_{j\in\Z}$. 
\begin{prop}[cf. {\cite[Chapter~3]{0768.35059}} ]\label{clasic}
Let $v=\{v_j\}_{j\in\Z}$ and $w=\{w_j\}_{j\in\Z}$ be two sequences in $l^1(\Z)\cap l^\infty(\Z)$. Any monotone numerical scheme \eqref{eq:scheme} which can be written in conservation form satisfies the following properties:
\begin{enumerate}
\item It is a contraction for $\|\cdot\|_{1,\Delta}$, that is:
	\begin{equation*}
		\|H_\Delta(v)-H_\Delta(w)\|_{1,\Delta} \le \|v-w\|_{1,\Delta}.
	\end{equation*}
\item It is $L^\infty$-stable, that is:
	\begin{equation*}
		\|H_\Delta(v)\|_{\infty,\Delta} \le \|v\|_{\infty,\Delta}.
	\end{equation*}
\item It preserves the sign, that is, if $v\ge 0$ then $H_\Delta(v)\ge0$.
\end{enumerate}
\end{prop}

Another important property that we need in order to prove the asymptotic behavior of the numerical scheme is the OSLC. Let us introduce
\begin{equation*}
	D^n=\sup _{j\in\Z}\left(\frac{u^n_{j+1}-u^n_j}{\Delta x}\right)^+
\end{equation*}
where $z^+:=\max\{0,z\}$.
\begin{defin}[cf. \cite{MR923922}] A numerical scheme \eqref{eq:scheme} is said to be OSLC consistent if:
	\begin{equation}\label{oslc1}
		D^n \le\frac{D^0}{1+ n \Delta t D^0}, \quad n\geq 1.
	\end{equation}
In particular, if a numerical scheme is OSLC consistent, it satisfies that
\begin{equation}\label{oslc2}
	\frac{u_{j+1}^n-u^n_{j-1}}{2\Delta x}\leq \frac 2{n\Delta t}, \quad n\geq 1.
\end{equation}
\end{defin}
We emphasize that, to the best of our knowledge, there is no general result stating wether a numerical scheme satisfies the OSLC or not. Nonetheless, there are some well-known schemes that have already been proved to be OSLC consistent (see \cite{MR923922}) on which we concentrate. In the sequel, we say that a scheme satisfies the OSLC when \eqref{oslc2} holds.

The analysis in this paper is limited to the following three 3-point schemes, with their numerical fluxes respectively:
\begin{enumerate}
	\item Lax-Friedrichs
		\begin{equation}\label{lf.numflux}
			g^{LF}(u,v)=\frac{u^2+v^2}{4}-\frac{\dx}{\dt}\left(\frac{v-u}{2}\right),
		\end{equation}
	\item Engquist-Osher
		\begin{equation}\label{eo.numflux}
			g^{EO}(u,v)=\frac{u(u+|u|)}{4}+\frac{v(v-|v|)}{4},
		\end{equation}
	\item Godunov
		\begin{equation}\label{god.numflux}
			g^{G}(u,v)=\begin{cases} \min\limits_{w\in[u,v]} \frac{w^2}{2},&\mbox{ if }u\le v, \\ \max\limits_{w\in[v,u]} \frac{w^2}{2},&\mbox{ if }v\le u. \end{cases}
		\end{equation}
\end{enumerate}

\begin{prop}[cf. \cite{MR923922}]\label{teo.oslc.schemes}
	Assuming that the CFL condition $\lambda \|u^n\|_{\infty,\Delta} \le 1$ is fulfilled, the Lax-Friedrichs, Engquist-Osher and Godunov schemes are monotone and OSLC consistent.
\end{prop}

In the case of the three numerical schemes above, thanks to the OSLC, we obtain that the $l^p_\Delta$-norms of the solutions decay similarly as in the continuous case.

\begin{prop}\label{decay}Let us consider a monotone conservative numerical scheme that is OSCL consistent. For any $p\in[1,\infty]$, there exists a constant $C_p>0$ such that the following holds
\begin{equation}\label{est.1}
	\|u^n\|_{p,\Delta}\leq C_p (n\Delta t)^{-\frac 12(1-\frac 1p)} \|u^0\|_{1,\Delta}^{\frac{1}{2}(1+\frac{1}{p})}, \quad \forall n\ge1.
\end{equation}
\end{prop}

\begin{proof}
Estimate \eqref{est.1} for $p=1$ follows from the fact that the scheme is conservative and, for $1<p<\infty$, it follows by applying H\"older's  inequality  once \eqref{est.1} is proved for $p=\infty$. Moreover, by the comparison principle, it is sufficient to consider the case of nonnegative initial data $u^0$. 

Let us now prove \eqref{est.1} for $p=\infty$ and nonnegative initial data. By the maximum principle, $u^n_j$ is nonnegative for all $j\in \Z$ and $n\geq 0$. We use now the OSLC \eqref{oslc2}. For all integers $m\geq 1$ we have
\begin{equation}\label{m}
	\frac{u_{j+2m}^n-u^n_{j}}{2\Delta x}\leq \frac {2m}{n\Delta t}, \quad n\geq 1.
\end{equation}
For a fixed $n$, let us now assume that the point $j$ where $u_j^n$ attains its maximum is even, the treatment of the other case being analogous,
\begin{equation*}
	u_{2j_0}^n:=\max _{j\in \Z} u_j^n. 
\end{equation*}
Hence, in view of \eqref{m} we get
\begin{equation*}
	u_{2j}^n\geq u_{2j_0}^n-4(j_0-j)\frac{\Delta x}{n\Delta t}, \quad \forall \, j\leq j_0.
\end{equation*}
Let us set
\begin{equation*}
	\gamma=j_0- \frac{n\Delta t}{4\Delta x} u_{2j_0}^n.
\end{equation*}
The last inequality and the mass conservation property imply that
\begin{align*}
	\|u^0\|_{1,\Delta}&=\Delta x\sum _{j\in \Z}u_j^0=\Delta x\sum _{j\in \Z}u_j^n\geq \Delta x\sum _{j\in\Z}u_{2j}^n\geq \Delta x\sum _{j=\lfloor\gamma\rfloor+1}^{j_0}u_{2j}^n \\
	&\geq \Delta x \sum _{j=\lfloor \gamma \rfloor+1}^{j_0} \Big( u_{2j_0}^n-4(j_0-j)\frac{\Delta x}{n\Delta t}\Big) = \frac{4(\Delta x)^2}{n\Delta t}\sum _{j=\lfloor \gamma\rfloor+1}^{j_0}( j- \gamma),
\end{align*}
where $\lfloor \gamma\rfloor$ denotes the largest integer less than or equal to $\gamma$. Since $\gamma\leq j_0$, it follows that
\begin{align*}
	\|u^0\|_{1,\Delta} & \geq \frac{4(\Delta x)^2}{n\Delta t}\sum _{j=\lfloor\gamma\rfloor+1}^{j_0} (j-\gamma)= \frac{2(\Delta x)^2}{n\Delta t} (j_0-\lfloor\gamma\rfloor)(j_0+\lfloor\gamma\rfloor+1-2\gamma)\\
	&= \frac{2(\Delta x)^2}{n\Delta t} (j_0-\gamma+\{\gamma\})(j_0-\gamma+1-\{\gamma\})\geq \frac{2(\Delta x)^2}{n\Delta t} (j_0-\gamma)^2=\frac{2(\Delta x)^2}{n\Delta t}(\frac {n\Delta t}{4\Delta x}u_{2j_0}^n)^2,
\end{align*}
where $\{\gamma\}=\gamma-\lfloor\gamma\rfloor\in[0,1)$. Hence, we obtain \eqref{est.1} for $p=\infty$:
\begin{equation}\label{estim1}
	\|u^n\|_{\infty,\Delta}\leq \frac {\sqrt{8}}{\sqrt {n\Delta t}}\|u^0\|_{1,\Delta}^{1/2}. 
\end{equation}
The proof is now finished.
\end{proof}

As in the context of the continuous hyperbolic conservation laws, the asymptotic profile of the numerical solutions need to satisfy another property. For any initial data $u_0\in L^1(\R)$, the solution of \eqref{eq:hyp.burgers} converges as the time $t$ goes to infinity to a N-wave $w_{p,q}$, determined by two quantities,
\begin{equation*}
	p=-\min_{x\in\R} \int_{-\infty}^x u_0(y)dy, \quad\quad\quad q= \max_{x\in\R} \int^\infty_x u_0(y)dy.
\end{equation*}
In fact, these parameters remain invariant for all time (e.g. \cite{Lax:1957}) and the same  should be expected at the discrete level. Let us remark that the mass $M$ of the solution of \eqref{eq:hyp.burgers} at each instant $t$ is $M=p+q$. We already know that the mass is also preserved at the discrete level.

\begin{theorem}\label{teo:sumcons}
	Assume that $u_0 \in L^1(\R)$, the CFL condition $\lambda \|u^n\|_{\infty,\Delta} \le 1$ is fulfilled and the numerical flux of a 3-point monotone conservative scheme as in \eqref{eq:consscheme} satisfies
\begin{subequations}
	\begin{align}
			g(\eta,\xi)=0,\quad \mbox{ when } \quad -1/\lambda\le\eta\leq 0\leq \xi\le1/\lambda, \label{hypsumcons1} \\
\intertext{and}
			\xi-\lambda g(\xi,-\xi)\ge0,\quad \mbox{ when } \quad 0\leq \xi\le1/\lambda \label{hypsumcons2}.
	\end{align}
\end{subequations}
Then, for any $n\ge0$ the following holds:
	\begin{equation}\label{eq:sumcons}
		\min_{k\in\Z}\sum_{j=-\infty}^k u^n_j=\min_{k\in\Z}\sum_{j=-\infty}^k u^0_j \quad \mbox{ and } \quad \max_{k\in\Z}\sum_{j=k}^\infty u^n_j=\max_{k\in\Z}\sum_{j=k}^\infty u^0_j.
	\end{equation}
\end{theorem}

We point out that both Engquist-Osher and Godunov schemes satisfy the hypothesis of this theorem, while Lax-Friedrichs does not. Indeed, for any $\eta, \xi$ such that $-1/\lambda\le\eta\leq 0\leq \xi\le1/\lambda$, we have that:
\begin{align*}
	g^{LF}(\eta,\xi)=0\quad \mbox{ if and only if }\quad \xi=\eta=0,
\end{align*}
and
\begin{align*}
	g^{EO}(\eta,\xi) =g^{G}(\eta,\xi) =0.
\end{align*}
Moreover, for any $0\le\xi\le1/\lambda$, the following holds:
\begin{align*}
	& \xi-\lambda g^{EO}(\xi,-\xi) =\xi-\lambda \xi^2=\xi(1-\lambda \xi)\ge0, \\
	& \xi-\lambda g^{G}(\xi,-\xi) =\xi-\lambda \frac{\xi^2}{2}=\xi(1-\lambda\frac{\xi}{2})\ge0.
\end{align*}
In the case of Engquist-Osher and Godunov schemes, property \eqref{eq:sumcons} will allow us to identify the asymptotic N-wave as in the continuous case \cite{Liu1984}. By contrary, for the Lax-Friedrichs scheme, the lack of the conservation of these quantities produces the loss of the N-wave shape and the appearance of the diffusive wave.

\begin{proof}For each $n\ge0$	we define
\begin{equation*}
	p_k^n:=\sum_{j=-\infty}^k u^n_j\quad \text{and}\quad q_k^n:=\sum_{j=k}^\infty u^n_j. 
\end{equation*}
Let $p^n$ and $q^n$ be the corresponding minimum and maximum of $\{p^n_k\}_{k\in\Z}$ and $\{q^n_k\}_{k\in\Z}$, respectively. It is easy to check that, according to \eqref{eq:consscheme}, $p_k^n$ and $q_k^n$ satisfy:
\begin{equation*}
	p^{n+1}_k=H_p(p^n_{k-1},p^n_k,p^n_{k+1}) \quad \quad \mbox{and}\quad \quad q^{n+1}_k=H_q(q^n_{k-1},q^n_k,q^n_{k+1}),
\end{equation*}
where
\begin{equation*}
	H_p(x,y,z):=y-\lambda g(y-x,z-y) \quad \quad \mbox{and}\quad \quad H_q(x,y,z):=y+\lambda g(x-y,y-z).
\end{equation*}

Let us fix $n\geq 0$ and assume that the minimum of $\{p^n_k\}_{k\in \Z}$ is attained at some index $K$, $p^n_{K}$. Then, it follows that 
\begin{equation*}
	u^n_K= p^n_K-p^n_{K-1}\le0\le p^n_{K+1}-p^n_{K}= u^n_{K+1}
\end{equation*}
and the maximum of $\{q_k^n\}_{k\in Z}$ is given by $q_{K+1}^n$. Thus, using \eqref{hypsumcons1}, we obtain that
\begin{equation*}
	p^{n+1} \le p^{n+1}_K = p^n_K -\lambda g(p_{K}^n-p_{K-1}^n,p^{n}_{K+1}-p_K^n)= p^n -\lambda g(p^n-p_{K-1}^n,p^{n}_{K+1}-p^n)=p^n
\end{equation*}
and
\begin{equation*}
	q^{n+1} \ge q^{n+1}_{K+1} = q^{n}_{K+1}+\lambda g(q_K^n-q^{n}_{K+1},q_{K+1}^n-q_{K+2}^n)= q^n+\lambda g(q_K^n-q^{n},q^n-q_{K+2}^n) = q^n.
\end{equation*}
Therefore $p^n\leq p^0$ and $q^n\geq q^0$ for all $n\ge0$.

We will prove now the reverse inequalities $p^n\geq p^0$ and $q^n\leq q^0$, that will finish the proof. This will be done by an inductive argument. Assuming that $p^n_k\geq p^0$ for all $k\in \Z$ we show that $p^{n+1}_k\geq p^0$. Using the identities
\begin{equation*}
	p^{n+1}_k -p^0= p^n_k -p^0 -\lambda g(p_{k}^n-p_{k-1}^n,p^{n}_{k+1}-p_k ^n)=H_p(p^n_{k-1}-p^0,p^n_k-p^0,p^n_{k+1}-p^0)
\end{equation*}
and
\begin{equation*}
	q^{n+1}_k -q^0= q^n_k -p^0 +\lambda g(q_{k}^n-q_{k-1}^n,q^{n}_{k+1}-q_k ^n)=H_q(q^n_{k-1}-q^0,q^n_k-q^0,q^n_{k+1}-q^0)
\end{equation*}
it is enough to prove that 
\begin{equation}\label{ineq.p}
	H_p(x,y,z) \ge 0 \quad \text{for all}\quad x,y,z\geq 0 
\end{equation}
and
\begin{equation}\label{ineq.q}
	H_q(x,y,z) \le 0 \quad \text{for all}\quad x,y,z\leq 0.
\end{equation}
Let us first prove \eqref{ineq.p}. Set $y-x=u$ and $z-y=v$. Since $x,z\ge0$, we have that $y\ge u$ and $y\ge -v$. We deduce that
\begin{equation*}
	y\ge\max\{u,-v,0\}.
\end{equation*}
This means that 
\begin{equation*}
	H_p(x,y,z)=y-\lambda g(u,v)\ge \max\{u,-v,0\}-\lambda g(u,v):=F(u,v)
\end{equation*}
By the CFL condition, it is sufficient to prove that function $F$ is nonnegative on the set 
\begin{equation*}
	\Omega=\{(u,v)\in\R^2:\lambda|u|\le1,\lambda|v|\le1\}.
\end{equation*}
We distinguish four regions in $\Omega$, according on the sign of $u$ and $v$:
\begin{equation*}
	\begin{array}{lllll}
		\Omega_+^+=\{(u,v)\in\Omega:u,v\ge0\}, &&&& \Omega_-^-=\{(u,v)\in\Omega:u,v\le0\}, \\
		\Omega_+^-=\{(u,v)\in\Omega:u\le0\le v\}, &&&& \Omega_-^+=\{(u,v)\in\Omega:u\ge0\ge v\}.
	\end{array}
\end{equation*}
Thus, we have explicitly:
\begin{equation*}
	F(u,v)=
	\begin{cases}
		u-\lambda g(u,v), &\mbox{ if } (u,v)\in\Omega^+_+, \\
		-v-\lambda g(u,v), &\mbox{ if } (u,v)\in\Omega^-_-, \\
		-\lambda g(u,v), &\mbox{ if } (u,v)\in\Omega^-_+, \\
		\max\{|u|,|v|\}-\lambda g(u,v), &\mbox{ if } (u,v)\in\Omega^+_-. \\
	\end{cases}
\end{equation*}
The monotonicity of the numerical flux $g$ guarantees that $F$ is increasing on $v$ in $\Omega^+_+$, decreasing on $u$ in $\Omega_-^-$, decreasing on $u$ and increasing on $v$ in $\Omega_+^-$. This shows that
\begin{equation*}
	\min_{\Omega}F=\min_{\Omega_-^+}F.
\end{equation*}
Using that in set $\Omega^+_-$ function $F$ is increasing on $u$ if $|u|\ge|v|$ and decreasing on $v$ if $|u|\le|v|$, we get
\begin{equation*}
	\min_{\Omega_-^+}F \ge\min_{0\le\xi\le1/\lambda} F(\xi,-\xi) =\min_{0\le\xi\le1/\lambda} \Big(\xi-\lambda g(\xi,-\xi)\Big).
\end{equation*}
The right-hand side of the above inequality is nonnegative due to hypothesis \eqref{hypsumcons2}. Therefore, $H_p$ satisfies \eqref{ineq.p} and, hence, $p^n=p^0$ for all $n\ge0$. Using a similar argument, the same result is proved for $q^n$, i.e. that $q^n=q^0$ for all $n\ge0$. The proof is now complete.
\end{proof}

To conclude this section, we present a second characterization of conservative monotone schemes, that better illustrate the artificial viscosity issue we are dealing with. The difference scheme \eqref{eq:scheme} is said to be in viscous form if there exists a function $Q:\R^{2k}\to\R$, called coefficient of numerical viscosity, such that
\begin{equation*}
	u_j^{n+1}=u_j^n-\lambda\left[\frac{(u_{j+1}^n)^2-(u_{j-1}^n)^2}{4}\right]+\frac{Q_{j+1/2}^n(u_{j+1}^n-u_j^n)-Q_{j-1/2}^n(u_{j}^n-u_{j-1}^n)}{2},
\end{equation*}
where
\begin{equation*}
	Q_{j+1/2}^n=Q(u^n_{j-k+1},\dots,u^n_{j+k}).
\end{equation*}

Three-point monotone schemes, for instance, can be always written in that way. For simplicity, when we treat the long time behavior of the numerical schemes, we rather prefer to put them in the following equivalent form
\begin{equation}\label{liv}
	\frac{u_{j}^{n+1}-u_j^n}{\Delta t}+\frac{(u_{j+1}^n)^2-(u_{j-1}^n)^2}{4\Delta x}=R(u_j^n,u_{j+1}^n)-R(u_{j-1}^n,u_{j}^n)
\end{equation}
where $R$ can be defined in a unique manner as
\begin{equation}\label{R}
	R(u,v)=\frac{Q(u,v)(v-u)}{2\Delta t}=\frac{1}{2\dx} \Big(\frac{u^2}{2}+\frac{v^2}{2}-2g(u,v)\Big).
\end{equation}

We recall that for the schemes considered in Theorem \ref{teo.oslc.schemes} we have
\begin{align}
	& R^{LF}(u,v)=\frac{v-u}{2\Delta t}, \notag\\
	& R^{EO}(u,v)=\frac{1}{4\dx}(v|v|-u|u|), \label{num.viscosity.coefs} \\[10pt]
	& R^{G}(u,v)=
	\begin{cases}
		\frac{1}{4\dx}\sign(|u|-|v|)(v^2-u^2),&v\le0\le u, \\[10pt]
		\frac{1}{4\dx}(v|v|-u|u|), & \mbox{elsewhere}.
	 \end{cases} 
	\notag
\end{align}


\section{Asymptotic Behavior}\label{section3}
This section is devoted to the proof of the main result of this paper, stated in Theorem \ref{asymptotic}, which describes the asymptotic profile developed by the numerical solutions of the schemes defined in Proposition \ref{teo.oslc.schemes}, that is, those satisfying the OSLC. Our analysis uses the method of self-similar variables, i.e., a rescaling of the solutions together with the compactness of the trajectories.

The key point in the analysis of the asymptotic behavior of the solutions of our numerical schemes is the degree of homogeneity of the term $R(u,v)$. We assume that there exists a real number $\alpha$ such that for any $u,v\in\R$ and $\mu>0$, function $R$ satisfies
\begin{equation}\label{degreeR}
	R(\mu u, \mu v)=\mu ^\alpha R(u,v).
\end{equation}
From \eqref{num.viscosity.coefs}, it is clear that $\alpha^{LF}=1$ for Lax-Friedrichs, while $\alpha^{G}=\alpha^{EO}=2$ for Godunov and Engquist-Osher, respectively.


\subsection{The piecewise constant solution}

In order to pass to the limit when doing the scaling argument, we first need to obtain bounds on the piecewise constant function $u_\Delta$, the piecewise constant interpolation \eqref{udelta} of $\{u^n_j\}_{j\in\Z,n\ge0}$ solution of scheme \eqref{liv}, in some Lebesgue spaces. Let us now apply the results of Section \ref{section2} to $u_\Delta$. It follows from \eqref{liv} that it satisfies the following equation:
\begin{equation}\label{sys.nonlocal}
\begin{cases}
	\frac{u_\Delta(t+\Delta t,x)-u_\Delta(t,x)}{\Delta t}+\frac{(u_\Delta(t,x+\Delta x))^2-(u_\Delta(t,x-\Delta x))^2}{4\Delta x}&\\
	\quad = R(u_\Delta(t,x),u_\Delta(t,x+\Delta x))-R(u_\Delta(t,x-\Delta x), u_\Delta (t,x)),& t\ge0,\mbox{ a.e. }x\in \R,\\
	u_\Delta(t,x)=u^0_\Delta(x),& t\in [0,\Delta t). 
\end{cases}
\end{equation}

The following Lemma gives us  the first bound on the solution $u_\Delta$. In the sequel, for any functions $f$ and $g$, we will write $f\lesssim g$ if there exists a constant $C>0$ such that $f\le Cg$.

\begin{lemma}\label{uinfinit}
There exists a positive constant $C=C(\dt,\|u_0\|_{1,\Delta})$ such that the following holds
\begin{equation*}
	\|u_\Delta(t)\|_{L^\infty(\rr)}\leq \frac C{\sqrt t},\quad \forall t>\dt.
\end{equation*}
\end{lemma}

\begin{proof}
From Proposition \ref{decay} we know that for any $n\geq 1$ the following holds
\begin{equation*}
	\|u^n\|_{\infty,\Delta}\leq \frac{C}{\sqrt {n\Delta t}} \|u^0\|_{1,\Delta}.
\end{equation*}
Let us now consider $t\in [n\dt,(n+1)\dt)$ with $n\geq 1$. Then
\begin{equation*}
\|u_\Delta(t)\|_{L^{\infty}(\rr)}=\|u^n\|_{\infty,\Delta}\leq \frac C{\sqrt {n\Delta t}}\leq  \frac {2C}{\sqrt {(n+1)\Delta t}}\leq  \frac {2C}{\sqrt t},
\end{equation*}
which proves the desired inequality.
\end{proof}

For the simplicity of the presentation, from now on we will denote by $\omega(h)$  the $L^1(\rr)$-modulus of continuity of the initial data $u^0_\Delta$:
\begin{equation*}
	\omega(h)=\int_{\rr}|u_\Delta^0(x+h)-u_\Delta^0(x)|dx.
\end{equation*}

\begin{lemma}\label{translation.x}
The solution of system \eqref{sys.nonlocal} satisfies 
\begin{equation*}
	\int _{\rr} |u_\Delta(t,x+h)-u_\Delta(t,x)|dx\leq \omega(h)
\end{equation*}
for all $h>0$ and $t>0$.
\end{lemma}

\begin{proof}
Let us consider $k\in \Z$ such that $k\Delta x\leq h<(k+1)\Delta x$. Then for any piecewise constant function $v$ as in \eqref{udelta}, we have
\begin{align*}
	\int _{\R}&|v(x+h)-v(x)|dx=\sum _{j\in \Z} \int _{x_{j-1/2}}^{x_{j+1/2}}|v(x+h)-v(x)|dx\\
	&= \sum _{j\in \Z} \int _{x_{j-1/2}}^{x_{j-1/2}+(k+1)\Delta x-h}|v(x+h)-v(x)|dx+\int _{x_{j-1/2}+(k+1)\Delta x-h}^{x_{j+1/2}}|v(x+h)-v(x)|dx\\
	&=((k+1)\Delta x-h)\sum _{j\in \Z}|v_{j+k}-v_j|+(h-k\Delta x)\sum _{j\in \Z} |v_{j+k+1}-v_j|.
\end{align*}
Applying this property to function $u_\Delta$ and using that for any $k\geq 1$ (cf. Proposition \ref{clasic})
\begin{equation*}
	\sum _{j\in \Z}|u^n_{j+k}-u^n_j|\leq \sum _{j\in \Z}|u^0_{j+k}-u^0_j|
\end{equation*}
we obtain that
\begin{align*}
	\int _{\rr} |u_\Delta(t,x+h)-&u_\Delta(t,x)|dx\\
	&\leq ((k+1)\Delta x-h)\sum _{j\in \Z}|u^0_{j+k}-u^0_j|+(h-k\Delta x)\sum _{j\in \Z} |u^0_{j+k+1}-u^0_j|\\
	&=\omega(h).
\end{align*}
This proves the desired result.
\end{proof}


\subsection{The rescaled solutions}
Let us now introduce for any $\mu>0$ the family of rescaled solutions
\begin{equation*}
	\umu(t,x)=\mu u_\Delta(\mu^2 t,\mu x), \quad t\geq 0,x\in \rr.
\end{equation*}
It follows that $\umu$ is piecewise constant on time intervals of length $\Delta t/\mu^2$ and space intervals of length $\Delta x/\mu$. Moreover, it satisfies the system
\begin{equation}\label{sys.umu}
	\begin{cases}
		\frac{\mu^2}{\Delta t}\Big( {\umu( t+\frac{\Delta t}{\mu^2},x)-\umu (t,x)}\Big)+\frac{\mu}{4\Delta x} \Big({(\umu(t,x+\frac{\Delta x}\mu))^2-
		(\umu(t,x-\frac{\Delta x}\mu))^2} \Big)&\\
		\quad =\mu ^{1-\alpha}\Big(\mu^2 R(\umu(t,x),\umu(t,x+\frac{\Delta x}\mu))-\mu^2 R(\umu(t,x-\frac{\Delta x}\mu), \umu (t,x))\Big),& t\ge0,\mbox{ a.e. } x\in\R,\\
		u^\mu_\Delta(t,x)=\mu u^0_\Delta(0,\mu x),& t\in [0,\Delta t/\mu^2). 
	\end{cases}
\end{equation}
The following lemmas will guarantee the convergence of the trajectories $\{u^\mu(t)\}_{\mu>0}$ as $\mu\to\infty$.
\begin{lemma}\label{space.est}
The solution of system \eqref{sys.umu} satisfies the following two estimates
\begin{enumerate}
	\item There exists a positive constant $C$ independent of $\mu$ such that 
		\begin{equation}\label{est.umu.inf}
			\|u^\mu(t)\|_\infty\leq \frac{C} { \sqrt t}, \quad \forall t>\frac{\dt}{\mu^2}.
		\end{equation}
	\item For all $h>0$ and $t>0$ the following holds
		\begin{equation*}
			\int _{\rr} |\umu(t,x+h)-\umu(t,x)|dx\leq \omega(h).
		\end{equation*}
\end{enumerate}
\end{lemma}

\begin{remark}
On the interval $[0,\dt/\mu^2]$ we have the rough estimate
\begin{equation}\label{est.mala}
\|u^\mu(t)\|_\infty=\mu \|u^0_\Delta \|_\infty\leq \frac{\mu}{\dx} \|u^0_\Delta\|_{L^1(\rr)}.
\end{equation}

\end{remark}

\begin{proof}The first estimate is a consequence of Lemma \ref{uinfinit}, while the second one follows from Lemma \ref{translation.x}.
\end{proof}

\begin{lemma}\label{time.est}
For any $0<t_1<t_2$, there exists a positive constant $C$ such that the  
\begin{equation*}
	\sup _{t\in [t_1,t_2]} \int _{\rr}|\umu (t+h,x)-\umu (t,x)|dx\leq C(h^{1/3}+\frac{h^{2/3}}{\sqrt{t_1}})\|u_0\|_{L^1(\rr)}+\omega (h^{1/3})
\end{equation*}
holds for any $h>0$ and $\mu>\sqrt{\frac{\dt}{t_1}}$.
\end{lemma}

\begin{proof}
We proceed as in \cite{1052.35126,MR2138795}. Since $\umu$ is piecewise constant in time, it is sufficient to consider the case when $t_1,t_2\in ({\Delta t}/{\mu^2})\Z$ and $h=k\Delta t/\mu^2$ with $k\in \Z$, $k\geq 1$. Let us set $t'=t+h-\Delta t/\mu^2.$ Then, for $t\in ({\Delta t}/{\mu^2})\Z $ we have
\begin{equation*}
	\umu(t+h)-\umu(t)=\frac{\mu^2 }{\Delta t} \int _{t}^{t'}\Big(\umu (s+\frac{\Delta t}{\mu^2})-\umu(s)\Big)ds.
\end{equation*}
Let us choose $\phi$ a smooth, bounded function on $\rr$. Multiplying \eqref{sys.umu} by $\phi$ and integrating in time and space we get
\begin{align}
	\int _{\rr}(\umu(t+h,x)-&\umu(t,x))\phi(x)dx=I_1+I_2= \notag\\
	=&\frac {\mu }{4\Delta x}\int _{t}^{t'} \int _{\rr}(\umu(s,x))^2\Big(\phi(x+\frac{\Delta x}{\mu})-\phi(x-\frac{\Delta x}{\mu})\Big)dsdx \label{inter1}\\
	&+\mu^{1-\alpha}\int_t^{t'}\int_\R\mu^2 R\Big (\umu(s,x),\umu(s,x+\frac{\Delta x}\mu)\Big)\Big(\phi(x)-\phi(x+\frac{\Delta x}\mu)\Big)dsdx. \notag
\end{align}
We now evaluate $I_1$ and $I_2$. Observe that since $\mu^2>\dt/t_1$ then $t\geq t_1\geq \dt/\mu^2$ and estimate \eqref{est.umu.inf} applies
\begin{align}\label{inter2}
	|I_1|&\lesssim \|\phi '\|_{L^\infty(\rr)} \int _{t}^{t'} \int _{\rr}(\umu(s,x))^2dsdx\lesssim {\|\phi '\|_{L^\infty(\rr)}} \int _{t}^{t'} \|\umu(s)\|_{L^\infty(\rr)}\int _{\rr}|\umu(s,x)|dsdx\\
	&\lesssim \|\phi '\|_{L^\infty(\rr)}\|u_0\|_{L^1(\rr)} { \int _{t}^{t'} \frac{ ds}{ \sqrt s}}\lesssim \|\phi '\|_{L^\infty(\rr)}\|u_0\|_{L^1(\rr)} \frac{h}{\sqrt t}. \notag
\end{align}
In the case of $I_2$, we use that $R(u,v)$ satisfies $R(u,v)=(v-u)/(2\Delta t)$ for the Lax-Friedrichs scheme and $|R(u,v)|\lesssim |u|^2+|v|^2$ for Engquist-Osher and Godunov schemes. For Lax-Friedrichs ($\alpha^{LF}=1$) we have for all $t>0$ that
\begin{align}\label{inter3}
	I_2&=\int_t^{t'}\int _{\rr} \frac{\mu^{2}}{2\Delta t} \umu(s,x)\Big(\phi(x+\frac{\Delta x}\mu)-2\phi(x)+\phi(x-\frac{\Delta x}\mu)\Big)dxds\\
	&\lesssim \|\phi ''\|_{L^\infty(\rr)}\int_t^{t'}\int _{\rr} |\umu(s,x)|dxds\leq h \|\phi ''\|_{L^\infty(\rr)}\|u_0\|_{L^1(\rr)}. \notag
\end{align}
For Engquist-Osher and Godunov schemes ($\alpha^{EO}=\alpha^G=2$), we have  similar estimates as in the case of $I_1$:
\begin{align}\label{inter4}
	I_2&\lesssim\|\phi '\|_{L^\infty(\rr)} \int _{t}^{t'} \int _{\rr}(\umu(s,x))^2dsdx\lesssim \|\phi '\|_{L^\infty(\rr)}\|u_0\|_{L^1(\rr)} \frac{h}{\sqrt t}.
\end{align}
Plugging estimates \eqref{inter2}, \eqref{inter3} and \eqref{inter4} into \eqref{inter1}, we obtain that
\begin{equation*}
	\int _{\rr}(\umu(t+h,x)-\umu(t,x))\phi(x)dx\lesssim \|u_0\|_{L^1(\rr)} \Big(h \|\phi ''\|_{L^\infty(\rr)}+\frac{h}{\sqrt t}\|\phi '\|_{L^\infty(\rr)}\Big).
\end{equation*}

Let us choose a mollifier $\rho$, a smooth nonnegative function supported in the interval $(-1,1)$ with unit mass, and take
\begin{equation*}
	\phi_h=h^{-1/3}\rho (h^{-1/3})\ast \Big[\sign \umu(t+h)-\sign \umu(t) \Big].
	\end{equation*}
We have that $|\phi_h '|\leq h^{-1/3}$, $|\phi_h ''|\leq h^{-2/3}$ and
\begin{equation}\label{ineg.100}
	\int _{\rr}(\umu(t+h,x)-\umu(t,x))\phi_h(x)dx\lesssim \|u_0\|_{L^1(\rr)} \Big(h^{1/3}+\frac{h^{2/3}}{\sqrt t}\Big). 
\end{equation}
Using that $\phi_h$ has unit mass  and that for any $a,b\in \rr$ we have $|a|-a\sign (b)\leq 2|a-b|$, we get
\begin{align*}
	|&\umu(t+h,x)-\umu(t,x)|- (\umu(t+h,x)-\umu(t,x))\phi_h(x)\\
	&=\int _{\rr} \phi_h(x-y)\Big(|\umu(t+h,x)-\umu(t,x)| \\
	&\qquad\qquad\qquad\qquad -(\umu(t+h,x)-\umu(t,x)) \big(\sign \umu(t+h,y)-\sign \umu(t,y) \big)\Big)dy\\
	&\leq 2\int _{\rr} \phi_h(x-y)\big|(\umu(t+h,x)-\umu(t,x)- ( u(t+h,y)- u(t,y) )\big|dy\\
	&\leq 2\int _{\rr} \phi_h(x-y)|\umu(t+h,x)- \umu(t+h,y)|dy+ 2\int _{\rr} \phi_h(x-y)|\umu(t,x)- \umu(t,y) |dy.
\end{align*}
Integrating the above inequality in $x$ we obtain that
\begin{align}\label{ineg.101}
	\int _{\rr}&|\umu(t+h,x)-\umu(t,x)|- (\umu(t+h,x)-\umu(t,x))\phi_h(x)dx\\
	\nonumber &\leq 2\int _{\rr} |\umu(t+h,x+h^{1/3})- \umu(t+h,x)|dx+ 2\int _{\rr}|\umu(t,x+h^{1/3})- \umu(t,x) |dx\\
	\nonumber &\leq 4\omega(h^{1/3}).
\end{align}
Combining \eqref{ineg.100} and \eqref{ineg.101} we obtain the desired result.
\end{proof}

\begin{lemma}\label{outside.R}
There exists a constant $C=C(\|u_0\|_{L^1(\rr)})$ such that 
\begin{equation}\label{est.out.R}
	\int_{|x|>2R} |\umu(t,x)|dx\leq \int _{|x|>R} |u_\Delta ^0|dx + C(\frac t{R^2}+\frac {\mu^{-1}+t^{1/2}}R).
\end{equation}
holds for any $t>0$, $R>0$ and  $\mu>1$.
\end{lemma}

\begin{proof} 
We first observe that it is sufficient to consider nonnegative initial data. Indeed, choosing $\tilde u^0=|u^0|$ as initial data in the numerical scheme, we have by the maximum principle that $|u^\mu(t,x)|\leq \tilde u^\mu(t,x)$ where $\tilde u^\mu$ is the solution that corresponds to the initial data $\tilde u^0$. It is then sufficient to prove estimate \eqref{est.out.R} for nonnegative initial data and solutions.

Let us now prove \eqref{est.out.R} for nonnegative solutions. Since $\umu$ is piecewise constant in time we consider the case $t=k\Delta t/\mu^2$, $k\in \Z$, $k\geq 1$, the case $k=0$ being obvious.
Let us choose $\rho\in C^\infty(\rr)$ such that $0\leq \rho\leq 1$ and
\begin{equation*}
	\rho=\begin{cases} 0, &|x|\leq 1,\\ 1, &|x|\geq 2. \end{cases} 
\end{equation*}
We set $\rho_R(x)=\rho(Rx)$. We multiply system \eqref{sys.umu} by $\rho_R$ and integrate on $(0,t')\times \rr$ where $t'=t-\Delta t/\mu^2$. The right hand side is given by
\begin{equation*}
	\frac{\mu^2}{\Delta t} \int _0^{t'} \int _{\rr} \Big(\umu (s+\frac{\Delta t}{\mu^2},x)-\umu(s,x)\Big)\rho_R(x)dxds = \int _{\rr} \umu (t,x) \rho_R(x)dx- \int _{\rr} \umu(0,x)\rho_R(x)dx.
\end{equation*}
Hence
\begin{align*}
	\int_{\rr}& \umu (t,x)\rho_R(x)dx-\int_{\rr} \umu (0,x)\rho_R(x)dx\\
	&=\int _0^{t'}\int _{\rr}(u^\mu (s,x))^2\Big[ \rho_R(x+\frac {\Delta x}{\mu})-\rho_R(x-\frac{\dx}\mu)\Big]\frac{\mu}{4\dx}dxds\\
	&\quad + \mu^{1-\alpha} \int _0^{t'}\int _{\rr} \mu^{2} R\Big(u^\mu(s,x),u^\mu(s+\frac{\dx}{\mu})\Big)\Big[\rho_R(x)-\rho_R(x+\frac{\dx}{\mu})\Big]dxds.\\
	&=I_1+I_2.
\end{align*}
In the first case using \eqref{est.umu.inf} and \eqref{est.mala} we get
\begin{align*}
	I_1&\lesssim \|\rho_R'\|_{L^\infty(\rr)} \int _0^{t'} \int_{\rr} (\umu(s,x))^2dsdx \lesssim \frac 1R\int_0^{t'}\|\umu (s)\|_{L^\infty(\rr)}\int_{\rr} |\umu(s,x)|dsdx \\
	&\leq \frac 1R \|u_0\|_{L^1(\rr)} ( \int _0^{\dt/\mu^2} \|\umu (s)\|_{L^\infty(\rr)}+\int_{\dt/\mu^2}^{t'} \frac 1{\sqrt s} ds)	\lesssim R^{-1} \|u_0\|_{L^1(\rr)} 
	(\frac 1{\mu}+ t^{1/2}).
\end{align*}
In the case of $I_2$, using the same argument as in Lemma \ref{time.est}, we get
\begin{equation*}
	I_2\lesssim \begin{cases} \frac{ \mu^{-1}+t^{1/2}}{R}\|u_0\|_{L^1(\rr)},& \alpha=2,\\[10pt]
	 \frac{t}{R^2} \|u_0\|_{L^1(\rr)},& \alpha=1. \end{cases}
\end{equation*}
It follows that 
\begin{align*}
	\int_{|x|>2R} \umu(t,x)dx&\leq\int _{|x|>\mu R}u_\Delta ^0(x)dx+C (\frac t{R^2}+\frac {\mu^{-1}+t^{1/2}}{R})\\
	&\leq \int _{|x|> R}u_\Delta ^0(x)dx+C (\frac t{R^2}+\frac {\mu^{-1}+t^{1/2}}{R}).
\end{align*}
The proof is now complete.
\end{proof}

\subsection{Passing to the limit} We are now in condition to prove the main result of this paper, stated in Theorem \ref{asymptotic}. The results obtained in the previous section will guarantee the compactness of the set $\{\umu\}_{\mu>0}$ needed to pass to the limit.

\begin{proof}[Proof of Theorem \ref{asymptotic}]
We proceed in several steps.

\textbf{Step I. Passing to the limit as $\mu\rightarrow \infty$.} From Riesz-Fréchet-Kolmogorov and Arzelà-Ascoli theorems and Lemmas \ref{space.est}, \ref{time.est} and \ref{outside.R}, we infer that $\{\umu\}_{\mu>0}$ is relatively compact in $C([t_1,t_2];L^1(\R))$ for any $0<t_1<t_2$. Consequently, there exist a subsequence, which we do not relabel, and a function $u^\infty\in C((0,\infty);L^1(\R))$ such that for any $0<t_1<t_2$
\begin{equation}\label{sequence}
	\umu \to u^\infty \mbox{ in } C([t_1,t_2];L^1(\R)) \mbox{ as } \mu\to\infty
\end{equation}
and
\begin{equation}\label{cont.ae}
	\umu (t,x)\to u^\infty (t,x), \quad a.e.\ (t,x)\in (0,\infty)\times \rr.
\end{equation}
Using the mass conservation of $u^\mu$ we obtain that
\begin{equation*}
	\int_\R u^\infty(t,x)dx=M_\Delta=\int _{\rr}u^0_\Delta(x)dx.
\end{equation*}
Moreover the almost everywhere convergence in \eqref{cont.ae} shows that there is a positive constant $C$ such that the limit function $u^\infty$ satisfies
\begin{equation}
	t^{\frac12}\|u^\infty(t)\|_{L^\infty(\R)}\leq C,\quad \forall t>0.
\end{equation}

We will now pass to the limit in the sense of distributions in equation \eqref{sys.umu}. Let us multiply it by a test function $\varphi\in C^\infty_c((0,\infty)\times\R)$ and integrate it both in space and time. The limit in the left-hand side is 
\begin{equation*}
	u^\infty_{t}+ \Big(\frac{(u^\infty)^2}{2}\Big)_x, \quad \text{in} \ \mathcal{D}((0,\infty)\times \R).
\end{equation*}
It remains to identify the limit for the right-hand side. Let us denote
\begin{align*}
	I_\mu:=\mu ^{1-\alpha}\int _{0}^\infty\int_\R\Big(\mu^2 R(\umu(t,x),\umu(t,x+\frac{\Delta x}\mu))-\mu^2 R(\umu(t,x-\frac{\Delta x}\mu), \umu (t,x))\Big)\varphi(t,x)dx.
\end{align*}
In the case of the Lax-Friedrichs scheme, $\alpha^{LF}=1$ and $R(u,v)=(v-u)/(2\Delta t)$. Thus
\begin{align*}
	I_\mu&=\frac{\mu ^{2}}{2\Delta t}\int _{0}^\infty\int_\R \umu(t,x)\Big(\varphi(t,x+\frac{\Delta x}\mu))-2 \varphi(t,x)+\varphi(t,x-\frac{\Delta x}\mu) \Big)dxdt\\
	&\longrightarrow \frac{(\dx)^2}{2\Delta t}\int _{0}^\infty \int _{\rr} u^\infty \varphi_{xx}dxdt, \quad \mbox{ as }\mu\to\infty.
\end{align*}
Hence the limit $u^\infty$ satisfies 
\begin{equation}\label{uLF}
	u^\infty_t + \Big(\frac{(u^\infty)^2}{2}\Big)_x= \frac{(\dx)^2}{2\Delta t} u_{xx}^\infty \quad \text{in}\ \mathcal{D'}((0,\infty)\times \rr).
\end{equation}

In the case of Engquist-Osher and Godunov schemes, $\alpha=2$. Using the explicit form of $R(u,v)$ given in \eqref{num.viscosity.coefs}, we obtain that
\begin{equation*}
	|R(u,v)|\lesssim | |u|u-|v|v|\leq |u-v|(|u|+|v|).
\end{equation*}
Assume that $\varphi$ is supported in the time interval $[t_1,t_2]$ with $t_1>0$. Then
\begin{align*}
	|I_\mu| & = \mu^{-1}\left|\int _{t_1}^{t_2}\int_\R \mu^2 R(\umu(t,x),\umu(t,x+\frac{\Delta x}\mu))\Big(\varphi(t,x)dx-\varphi(t,x+\frac{\Delta x}\mu)\Big)dx\right| \\
	& \lesssim \Delta x\|\varphi'\|_\infty \int _{t_1}^{t_2}\int_\R | R(\umu(t,x),\umu(t,x+\frac{\Delta x}\mu))|dxdt\\
	& \lesssim \Delta x\|\varphi'\|_\infty \int _{t_1}^{t_2} \int_\R |\umu(t,x)-\umu(t,x+\frac{\Delta x}\mu)| \big(|\umu(t,x)|+|\umu(t,x+\frac{\Delta x}\mu)|\big)dxdt \\
	& \lesssim \Delta x\|\varphi'\|_\infty C(t_1) \max_{t\in [t_1,t_2] } \int_\R |\umu(t,x)-\umu(t,x+\frac{\Delta x}\mu)| dx.
\end{align*}
Using Lemma \ref{time.est} we obtain that $I_\mu\rightarrow 0$ as $\mu\rightarrow\infty$. Therefore, the limit point $u^\infty$ satisfies
\begin{equation}\label{uEO}
	u^\infty_t + \Big(\frac{(u^\infty)^2}{2}\Big)_x= 0 \quad \text{in}\ \mathcal{D'}((0,\infty)\times \rr).
\end{equation}

Let us now recall that in view of the OSLC \eqref{oslc2} for any $t>\dt$ and a.e $x\in \rr$ we have
\begin{equation*}
	\frac {u_\Delta(t,x+\dx) -u_\Delta(t,x-\dx)}{2\dx}\leq \frac  C{t}.
\end{equation*}
Hence, for all $t>\dt/\mu^2$ we have
\begin{equation*}
	\frac{\mu^2}{2\dx} \Big(u^\mu(t,x+\frac \dx{\mu^2}) -u^\mu(t,x-\frac \dx{\mu^2})\Big ) \leq \frac  C{t}.
\end{equation*}
Letting $\mu\to\infty$, we obtain that for any $t>0$ the limit point $u^\infty$ satisfies
\begin{equation*}
	u^\infty_x(t)\leq \frac C{t} \quad \text{in}\ \mathcal{D}'(\rr).
\end{equation*}
This shows that $u^\infty$ is an entropy solution. Note that this can also be guaranteed by the monotonicity of the numerical schemes, as monotone schemes are consistent with any entropy condition \cite[Chapter~3]{0768.35059}.

\textbf{Step II. Initial data}. It remains to identify the behavior of $u^\infty$ as $t\to0$. We will prove that $u^\infty(t)\rightarrow M_\Delta\delta _0$ as $t\rightarrow 0$, in the sense of bounded measures, i.e.
\begin{equation}\label{conv.initial.data}
	\lim_{t\to0} \int_\R u^\infty(t,x)\varphi(x)dx= M_\Delta\varphi(0)
\end{equation}
 for every bounded continuous function $\varphi$. By a density argument it is sufficient to consider the case $\varphi\in C^\infty_c(\R)$. Thus we will conclude that $u_\infty(0)=M_\Delta\delta_0$ in the sense of bounded measures.

Let us choose $t=k \Delta t/\mu^2$, $k\in \Z$, $k\geq 1$. Then for any $k\geq 1$ (for $k=0$ it is obvious)
\begin{align*}
	\frac{\mu^2}{\Delta t}\int_0^t\int_\R\Big( \umu( s&+\frac{\Delta t}{\mu^2},x)-\umu (s,x)\Big)\varphi(x)dxds \\
	&= \frac{\mu^2}{\Delta t}\int_{t}^{t+\frac{\Delta t}{\mu^2}}\int_\R\umu(s,x)\varphi(x)dxds-\frac{\mu^2}{\Delta t}\int_0^{\frac{\Delta t}{\mu^2}}\int_\R\umu (s,x)\varphi(x)dxds \\
	&= \int_\R\umu(t,x)\varphi(x)dx-\int_\R\umu (0,x)\varphi(x)dx .
\end{align*}
Let us consider a $t$ such that $k\Delta t/\mu^2\leq t<(k+1)\Delta t/\mu^2$ for some $k\geq 0$. Since $\umu$ is piecewise constant, we have:
\begin{align*}
	 \int_\R\umu(t,x)&\varphi(x)dx-\int_\R\umu (0,x)\varphi(x)dx = \int_\R\Big(\umu(k\frac{\dt}{\mu^2},x)-\umu (0,x)\Big)\varphi(x)dx\\
	 &=\frac{\mu^2}{\Delta t} \int_0^{k\frac{\dt}{\mu^2}}\int_\R\Big( \umu( s+\frac{\Delta t}{\mu^2},x)-\umu (s,x)\Big)\varphi(x)dxds\\
	 &=\frac {\mu }{4\Delta x} \int_0^{k\frac{\dt}{\mu^2}} \int _{\rr}(\umu(s,x))^2\Big(\varphi(x+\frac{\Delta x}{\mu})-\varphi(x-\frac{\Delta x}{\mu})\Big)dsdx \\
	&\quad +\mu^{1-\alpha}\int_0^{k\frac{\dt}{\mu^2}} \mu^2 R\Big (\umu(s,x),\umu(s,x+\frac{\Delta x}\mu)\Big)(\varphi(x)-\varphi(x+\frac{\Delta x}\mu))dsdx. \notag
\end{align*}
Following the same steps as in Lemma \ref{time.est} and using that $k\Delta t/\mu^2\leq t$ we obtain that
\begin{align*}
	\Big| \int_\R\umu(t,x)\varphi(x)dx-\int_\R\umu (0,x)\varphi(x)dx\Big|\lesssim \|u^0_\Delta\|_{L^1(\R)} \big( \|\varphi' \|_{L^\infty(\R)}(\mu^{-1}+ t^{1/2})+ \|\varphi''\|_{L^\infty(\R)} t\big).
\end{align*}
Using the definition of $\umu(0,x)$ and letting $\mu\rightarrow \infty$ we get
\begin{equation*}
	\left|\int_\R u^\infty(t,x)\varphi(x)dx-\varphi(0)\Big(\int_\R u^0_\Delta(x)dx\Big)\right|\le C(\varphi)(t^{1/2}+t).
\end{equation*}
The proof of \eqref{conv.initial.data} is now complete.

\textbf{Step III. Identification of the limit.}
In the case of the Lax-Friedrichs scheme system (\ref{uLF}-\ref{conv.initial.data}) has a unique solution $w_M$ given by \eqref{wm}. Since $w_M$ is the is the only possible accumulation point of $\{\umu\}_{\mu>0}$ in $C((0,\infty),L^1(\R))$ as $\mu\to\infty$, the whole family converges to $w_M$. Therefore:
\begin{equation*}
	\lim_{\mu\to\infty} \|\umu(1)-w_M(1)\|_{L^1(\R)}=0
\end{equation*}
so, setting $\mu=t^{1/2}$, we recover \eqref{limit.infty} for $p=1$. H\"older's inequality and Proposition \ref{decay} allow us to deduce \eqref{limit.infty} for $p\in(1,\infty)$.

In the case of Engquist-Osher and Godunov schemes, as proved in \cite{Liu1984}, there are infinitely many solutions $w_{p_\Delta,q_\Delta}\in C((0,\infty),L^1(\rr))$ of system (\ref{uEO}-\ref{conv.initial.data}), so we have to identify the parameters $p_\Delta$ and $q_\Delta$. As pointed in \cite{Liu1984} it remains to
to identify the limit as $t\downarrow 0$ of 
\begin{equation*}
	v(t,x)=\int_{-\infty}^x u^\infty(t,y)dy.
\end{equation*}
Since $u^\infty$ converges to $M\delta_0$, we have:
\begin{align*}
	\lim_{t\to0}\int_{-\infty}^x u^\infty(t,y)dy=0,\ \forall x<0 \quad\mbox{ and }\quad \lim_{t\to0}\int_{-\infty}^x u^\infty(t,y)dy=M_\Delta,\ \forall x>0.
\end{align*}
It remains to determine the above limit when $x=0$. Note that that the map $t\rightarrow v(t,0)$ is increasing when $t\downarrow 0$ so that there exists 
\begin{equation*}
	-l=\lim _{t\downarrow 0} v(t,0).
\end{equation*}
This proves that $u^\infty=w_{l,l+M_\Delta}$. To finish the proof it remains to show that $l=p_\Delta$. According to \cite{Liu1984}, parameter $l$ is characterized by 
\begin{equation*}
	-l=\min_{x\in \rr} \int _{-\infty}^ xw_{l,l+M_\Delta}(t,y)dy=\min_{x\in \rr} \int _{-\infty}^ x u^\infty(t,y)dy.
\end{equation*}
So it is sufficient to prove that
\begin{equation*}
	\min _{x\in \rr} \int _{-\infty}^x u^\infty(t,y)dy=\min_{x\in\R} \int_{-\infty}^x u^0_\Delta(y)dy.
\end{equation*}
By Theorem \ref{teo:sumcons}, we know that
\begin{align*}
	\min_{x\in\R} \int_{-\infty}^x u^0_\Delta(y)dy = \min_{x\in\R} \int_{-\infty}^x u_\Delta(\mu^2t, y)dy = \min_{x\in\R} \int_{-\infty}^x \umu(t,y)dy.
\end{align*}
Since $\umu$ converges to $u^\infty$ in $L^1(\R)$, its primitive converges uniformly to the primitive of $u^\infty$ when $\mu\to \infty$. So we have
\begin{align*}
	\min_{x\in\R} \int_{-\infty}^x u^\infty(t,y)dy=\lim_{\mu\to\infty} \min_{x\in\R} \int_{-\infty}^x \umu(t,y)dy= \min_{x\in\R} \int_{-\infty}^x u^0_\Delta(y)dy =-p_\Delta.
\end{align*}
We conclude that $u^\infty$ is the unique solution $w_{p_\Delta,q_\Delta}$ to \eqref{limit2} with $p_\Delta$ and $q_\Delta$ as in Theorem \ref{asymptotic}. Since $w_{p_\Delta,q_\Delta}$ is the is the only possible accumulation point of $\{\umu\}$ in $C((0,\infty),L^1(\R))$ as $\mu\to\infty$, the whole family converges to $w_{p_\Delta,q_\Delta}$. Therefore:
\begin{equation*}
	\lim_{\mu\to\infty} \|\umu(1)-w_{p_\Delta,q_\Delta}(1)\|_{L^1(\R)}=0
\end{equation*}
so, setting $\mu=t^{1/2}$, we recover \eqref{limit.infty} for $p=1$. H\"older's inequality and Proposition \ref{decay} allow us to deduce assertion \eqref{limit.infty} for $p\in(1,\infty)$. This completes the proof of the main result of this paper.
\end{proof}


\section{Simulations}\label{section4}

On the following, we illustrate the main results of previous sections with some numerical simulations. Let us consider the inviscid Burgers equation
\begin{equation}\label{eq.burgers}
	u_t + \left(\frac{u^2}{2}\right)_x=0,\quad x\in\R,t>0
\end{equation}
with initial data
\begin{equation*}
	u_0(x)=
\begin{cases}
	-0.05,&x\in[-1,0],\\
	0.15,&x\in[0,2],\\
	0,&\mbox{elsewhere.}
\end{cases}
\end{equation*}
In this case, the parameters that describe the asymptotic N-wave profile, defined in \eqref{nwave}, are:
\begin{equation*}
	M=0.25\ , \quad p=0.05 \quad \mbox{and} \quad q= 0.3.
\end{equation*}
We focus our experiments on the schemes described in Proposition \ref{teo.oslc.schemes}: on the one hand, Engquist-Osher and Godunov schemes, as examples of well-behaving schemes, and on the other, the Lax-Friedrichs scheme.

For the spatial domain discretization, we take $\Delta x=0.1$ as the mesh size for the interval $[-350, 800]$. Let us remark that, in general, it is not possible to impose homogeneous Dirichlet boundary conditions on both sides of the interval (e.g. \cite{Bardos:1979}). Nevertheless, due to the finite speed of propagation, we can consider an  large enough domain to guarantee that the boundary conditions do not interfere on the solution. Regarding the time-step, we simply choose $\Delta t=0.5$, that verifies the CFL condition in the three cases.

In Figure \ref{fig:nwaves} we show the numerical solution obtained at time $t=10^5$. It is possible to appreciate how the numerical viscosity of Lax-Friedrichs has dissipated the negative part of the solution. After such a long time, it only remains  a diffusive positive profile, i.e., the wave described in the first case of Theorem \ref{asymptotic}. On the contrary, both Engquist-Osher and Godunov schemes preserve the shape of the N-wave. 

\begin{figure}[t!]
	\includegraphics[width=0.85\linewidth]{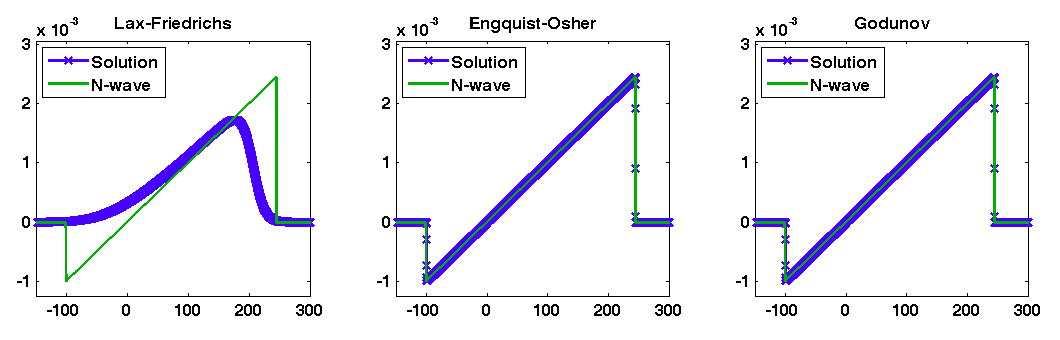}
	\caption{Solution to the Burgers equation at $t=10^5$ using Lax-Friedrichs (left), Engquist-Osher (center) and Godunov (right) schemes. The green line corresponds to the predicted N-wave, defined as in \eqref{nwave}.}
	\label{fig:nwaves}
\end{figure}

We can confirm this loss of the negative part of the solution for the Lax-Friedrichs scheme  in Figure \ref{fig:masses}. While the mass of the solution is conserved throughout time in the three cases, the Lax-Friedrichs scheme fails to preserve, in addition,  both the masses of the positive and negative parts respectively. Let us notice that when a solution crosses the horizontal axis just once, these masses are equivalent to the parameters $p$ and $q$ computed at each time.

\begin{figure}[t!]
	\includegraphics[width=0.85\linewidth]{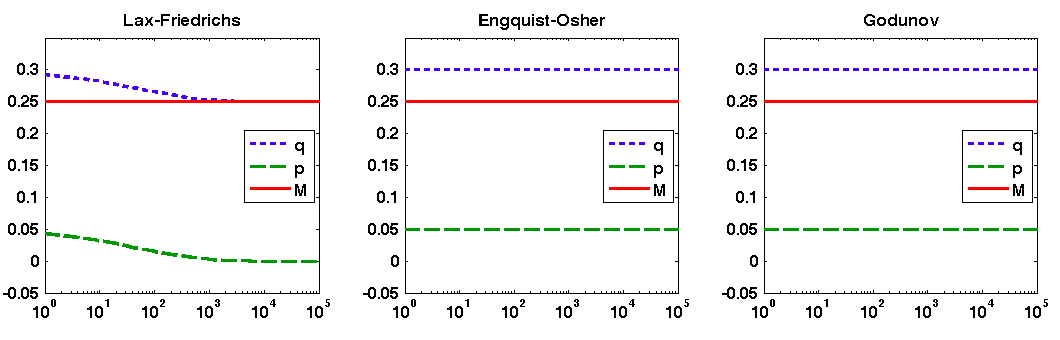}
	\caption{Evolution in time (on a logaritmic scale) of the total mass of the solution (continue), together with the positive (dotted) and negative (dashed) masses, using Lax-Friedrichs (left), Engquist-Osher (center) and Godunov (right) schemes.}
	\label{fig:masses}
\end{figure}

Finally, in Figure \ref{fig:norms} we show the evolution of the $L^1$ and $L^2$ norms of the difference between the numerical solution and the N-wave. This confirms that, for large times, the behavior of the solutions obtained by the Engquist-Osher and Godunov schemes are the expected ones. By the contrary, the performance of the Lax-Friedrichs scheme is far far from being correct.

\begin{figure}[t!]
	\includegraphics[width=0.85\linewidth]{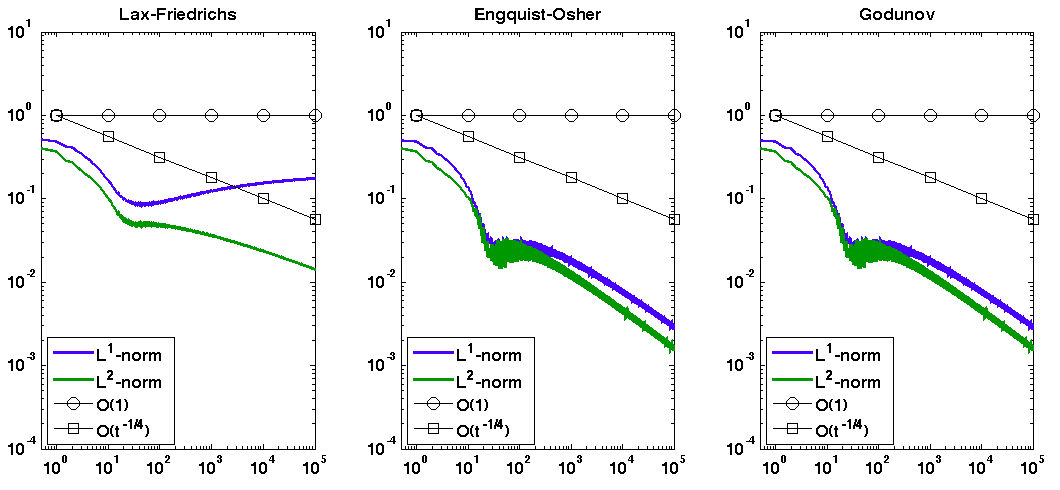}
	\caption{Evolution in time (on a logaritmic scale) of the $L^1$ and $L^2$ norms of the difference between the N-wave and the numerical solution given by Lax-Friedrichs (left), Engquist-Osher (center) and Godunov (right) schemes.}
	\label{fig:norms}
\end{figure}


\section{Similarity variables}\label{section5}

One of the disadvantages of the numerical approach we have developed in the previous sections is that the considered computational domain has to be extremely large, in comparison to the support of the initial data because of its time-spreading. In this section we use similarity variables, a classical tool at the continuous level and that, as we shall see, at the discrete one, leads to an alternate way of understanding the large time behavior and  to a significant decrease of the computational cost. In \cite{Kim2001} similarity variables were used to analyze the transition to the asymptotic states  to estimate the time of evolution from an N-wave to the final stage of a diffusion wave for the viscous Burgers equation. As we shall see, the same phenomena occurs for numerical schemes in case its effective asymptotic numerical positive is non-negligible as it occurs with the Lax-Friedrichs scheme.
 
Let us consider the change of variables given by:
\begin{equation*}
	s=\ln(t+1),\quad \quad \xi=x/\sqrt{t+1},\quad \quad w(\xi,s)=\sqrt{t+1}\ u(x,t),
\end{equation*}
which turns \eqref{eq.burgers} into
\begin{equation}\label{eq.burgers.sta}
	w_s + \left(\frac{1}{2}w^2-\frac{1}{2}\xi w\right)_\xi=0,\quad \xi\in\R,s>0.
\end{equation}
In this case, the asymptotic profile of the solutions is a N-wave as follows
\begin{equation}\label{nwave.sta}
	N_{p,q}(\xi)= \begin{cases} \xi,& -\sqrt{2p}<\xi<\sqrt{2q}, \\ 0,& \mbox{elsewhere}, \end{cases}
\end{equation}
where $p$ and $q$ are, respectively, the negative and positive mass of the initial data. 

The asymptotic profiles of the Burgers equation in the original variables become, in the similarity ones, steady state solutions. accordingly, in the similarity variables, the asymptotic convergence towards a self-similar solution in the self-similar ones becomes, simply, the convergence towards steady-states. One further  advantage of considering similarity variables is that the the support of the solutions does not grow indefinitely anymore and, thus, their numerical approximation is easier to handle.

\subsection{Presentation of discrete similarity schemes}
In our numerical analysis, we first need to adjust the three numerical schemes under consideration to the similarity variables. The general form of the scheme is still given in the conservative form
\begin{equation}\label{sta.scheme}
	w^{n+1}_j=w^n_j-\frac{\Delta s}{\Delta\xi}(g^n_{j+1/2}-g^n_{j-1/2}),\quad j\in\Z,n\ge0.
\end{equation}
From \cite{Kim2001}, the numerical flux for the Godunov scheme is given by
\begin{equation}\label{god.sta}
	g^n_{j+1/2}=
	\begin{cases}
		I(w^n_{j+1},\bar\xi),&\mbox{if } h(w^n_{j},\bar\xi)+h(w^n_{j+1},\bar\xi)\le0 \mbox{ and } h(w^n_{j+1},\bar\xi)\le0, \\
		I(w^n_{j},\bar\xi),&\mbox{if } h(w^n_{j},\bar\xi)+h(w^n_{j+1},\bar\xi)>0 \mbox{ and } h(w^n_{j},\bar\xi)>0, \\
		-3\bar\xi^2/8,&\mbox{if } h(w^n_{j},\bar\xi)<0 \mbox{ and } h(w^n_{j+1},\bar\xi)>0,
	\end{cases}
\end{equation}
where $\bar\xi=\xi_{j+1/2}$, $h$ is the wave speed
\begin{equation*}
	h(w,\xi)=w-\xi/2
\end{equation*}
and $I(w,\xi)$ is as follows
\begin{equation}\label{god.sta.flux}
	I(w,\xi)=\frac12w^2(e^{\Delta s}-1)-\xi w(e^{\Delta s/2}-1).
\end{equation}
For the Lax-Friedrichs scheme we take:
\begin{equation}\label{lf.sta}
	g^n_{j+1/2}=\frac{(w^n_j)^2-\bar\xi w^n_j+(w^n_{j+1})^2-\bar\xi w^n_{j+1}}{4}-\frac{\dxi}{\ds}\left(\frac{w^n_{j+1}-w^n_j}{2}\right),
\end{equation}
while for Engquist-Osher we choose:
\begin{equation}\label{eo.sta}
	g^n_{j+1/2}=\frac{(w^n_j-\bar\xi/2)(w^n_j-\bar\xi/2+|w^n_j-\bar\xi/2|)}{4}+\frac{(w^n_{j+1}-\bar\xi/2)(w^n_{j+1}-\bar\xi/2-|w^n_{j+1}-\bar\xi/2|)}{4}-\frac{\bar\xi^2}{8}.
\end{equation}

The advantage of using numerical schemes in these similarity variables is that we do not need to cover large domains, neither in time nor space, to capture the dynamics of solutions. In Figure \ref{fig:mesh} we transform a rectangular mesh on the space-time domain $[-3,3]\times[0,3]$ for $(\xi,s)$ into the corresponding parabolic mesh on $(x,t)$. We can observe that computations done for equation \eqref{eq.burgers.sta} in short periods of time (for instance, up to $s=4$) are equivalent to large-time solutions in the original equation \eqref{eq.burgers} ( $t\approx53$ in the example under consideration). 

\begin{figure}[t!]
	\includegraphics[width=0.85\linewidth]{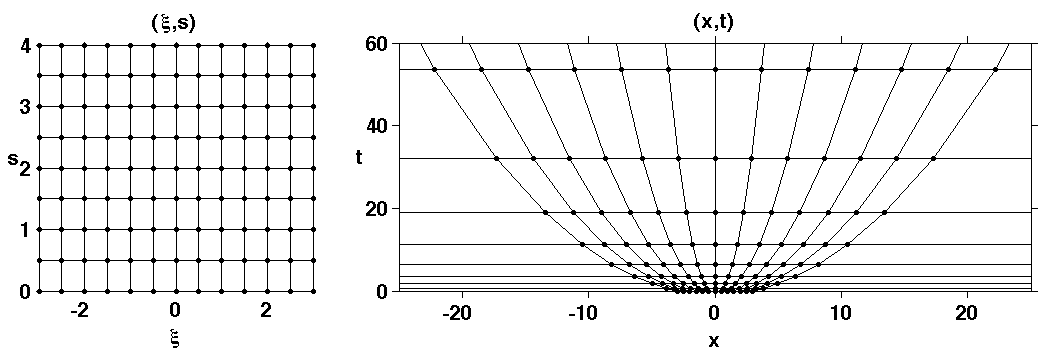}
	\caption{Comparison between the mesh on variables $(\xi,s)$ and $(x,t)$. The rectangular mesh $[-3,3]\times[0,4]$ for $(\xi,s)$ covers a trapezoidal domain that reaches $x\in[-22,22]$ at $t\approx53$.}
	\label{fig:mesh}
\end{figure}

\subsection{Discussion on discrete steady states}
For the numerical approximation schemes above in the similarity variables we expect a similar behavior as in the continuous case. Namely, that numerical solutions as the discrete self-similar time evolves converge towards numerical steady state solutions. Of course we expect this steady state solutions to converge towards the corresponding continuous ones as the mesh-size tends to zero. A complete analysis of these issues is out of the scope of this paper. As we shall see the numerical experiments confirm this fact establishing once more a clear distinction between the Lax-Friedrichs scheme that behaves in a parabolic manner and the two others.

To better understand the nature of the steady-state solutions, observe that those  of  \eqref{eq.burgers.sta} satisfy:
\begin{equation}
	\Big(\frac{w^2-\xi w}{2}\Big)_\xi =0,\quad \xi\in\R,s>0.
\end{equation}
Since the solution must vanish on the tails, we deduce that 
\begin{equation}\label{steady}
	\frac12 w^2-\frac12 \xi w=0.
\end{equation}
Thus, either $w=0$ or $w=\xi$. Wether to choose one or the other is decided using entropy conditions and the conservation of $p$ and $q$, as in the case of equation \eqref{eq.burgers} (cf. \cite{Kim2001,Liu1984}). The obtained profiles are, precisely, those given by \eqref{nwave.sta}. 

On the other hand, the steady-state solution for the viscous version (with some viscosity $\eps>0$) satisfies
\begin{equation}\label{steady.visc}
	-\eps w_{\xi} + \frac{w^2-\xi w}{2} =0,\quad \xi\in\R,s>0
\end{equation}
which is not an algebraic equation anymore, but an ODE.

Similarly, a steady-state solution $\bar w=\{\bar w_j\}_{j\in\Z}$ for \eqref{sta.scheme}, if it exists, must satisfy that
\begin{equation*}
	g(\bar w_j,\bar w_{j+1},\bar \xi_{j+1/2})=0\quad \forall j\in\Z.
\end{equation*}

In the case of the Godunov scheme \eqref{god.sta}, we can formally deduce from \eqref{god.sta.flux} that the asymptotic profile can only take values
\begin{equation*}
	\bar w_j=0 \quad \mbox{ or } \quad \bar w_j=\frac{2}{e^{\ds/2}+1}\bar\xi \quad \mbox{ or } \quad \bar w_{j+1}=\frac{2}{e^{\ds/2}+1}\bar\xi,
\end{equation*}
that is, $\bar w$ can just be 0, linear or a combination of both. Let us observe that the slope of the latter is not the same as the one of the continuous model, but tends to it when $\ds\to0$. Note that this is compatible with a closed form of the numerical flux of Godunov where the inequalities in \eqref{god.sta.flux} are not strict.

The nature of the steady of the Engquist-Osher scheme \eqref{eo.sta} is slightly different. We have:
\begin{equation*}
	\frac12\left(\frac{(\bar w_j-\bar\xi/2)(\bar w_j-\bar\xi/2+|\bar w_j-\bar\xi/2|)}{2}+\frac{(\bar w_{j+1}-\bar\xi/2)(\bar w_{j+1}-\bar\xi/2-|\bar w_{j+1}-\bar\xi/2|)}{2}\right)-\frac{\bar\xi^2}{8}=0.
\end{equation*}
This is an upwind discretization of equation \eqref{steady}, rewritten as follows:
\begin{equation*}
	\frac{(w-\xi/2)^2}{2}-\frac{\xi^2}{8} =0,\quad \xi\in\R,s>0.
\end{equation*}
Therefore, the expected large-time behavior of the numerical simulation is, again, similar to the one of the continuous equation. 

On the contrary, the Lax-Friedrichs scheme \eqref{lf.sta} we have that:
\begin{equation*}
	-\frac{\dxi}{\ds}\left(\frac{\bar w_{j+1}-\bar w_j}{2}\right)+\frac12\left(\frac{(\bar w_j)^2-\bar\xi \bar w_j}{2}+\frac{(\bar w_{j+1})^2-\bar\xi \bar w_{j+1}}{2}\right)=0.
\end{equation*}
We can distinguish two terms: the second is an average of the the flux, but the first term corresponds to an artificial viscosity, as in \eqref{steady.visc}. Thus, we can expect diffusivity at large times that distort the asymptotic N-wave.

\subsection{Numerical example}
In the following example we compare the behavior of the numerical solutions directly with the asymptotic profile of the continuous solution of \eqref{eq.burgers.sta}. Let us choose initial data
\begin{equation}
	w_0(x)= \begin{cases} x+10,& -12<x<-8, \\ x,& -\sqrt{2}<x<\sqrt{6}, \\ 0,& \mbox{elsewhere}, \end{cases},
\end{equation}
which corresponds to two separated N-waves. From the continuous point of view, at the beginning the first one moves towards the origin until it collides with the other one. Then they both interact, resulting on a new N-wave which is similar to the expected asymptotic profile. The same behavior should be required for the numerical schemes, but, as we show in Figures \ref{fig:sta_sol1}, \ref{fig:sta_sol2} and \ref{fig:sta_sol3}, the performance may vary depending on the chosen numerical flux.

We consider the mesh size $\Delta \xi=0.01$ and a time step $\Delta s=0.0005$, which is small enough to satisfy the CFL condition. Let us recall that, since the support of the solution remains in a bounded interval, we can choose a small spatial domain. In the first part of the simulation, the three numerical schemes behave in the correct manner, as the two N-waves collapse into one. The first column in the figures shows this regime.

Once the unique N-wave takes form, the behavior of the schemes takes different paths. Both Godunov and Engquist-Osher schemes maintain the N-wave shape that gradually converges to the hyperbolic asymptotic profile, as we can appreciate in the second column of Figures \ref{fig:sta_sol1} and \ref{fig:sta_sol2}.

\begin{figure}[t!]
	\includegraphics[width=0.85\linewidth]{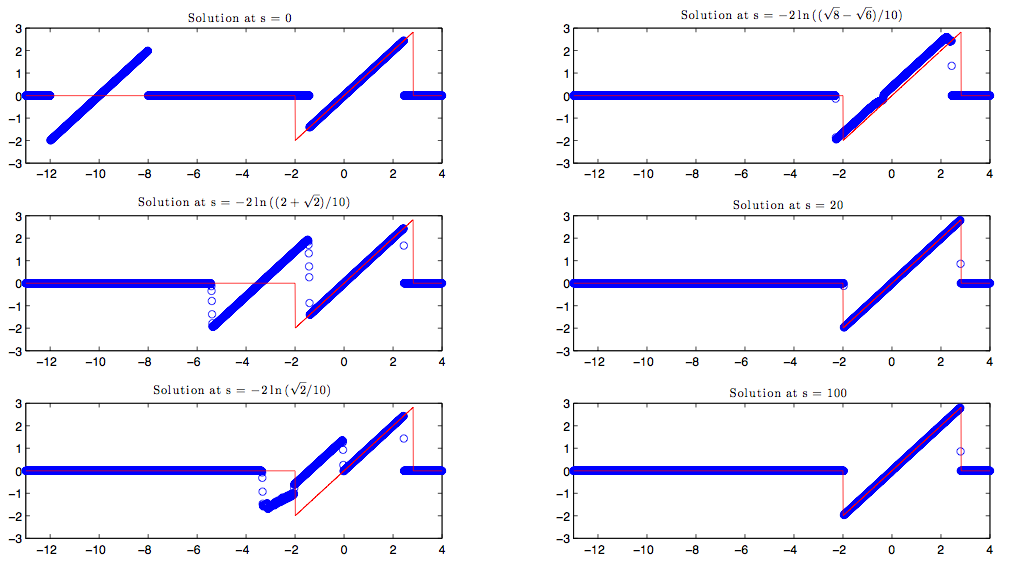}
	\caption{Convergence of the numerical solution of \eqref{eq.burgers.sta} using Godunov scheme (circle dots) to the asymptotic N-wave (solid line). We take $\Delta \xi=0.01$ and $\Delta s=0.0005$.}
	\label{fig:sta_sol1}
\end{figure}

\begin{figure}[t!]
	\includegraphics[width=0.85\linewidth]{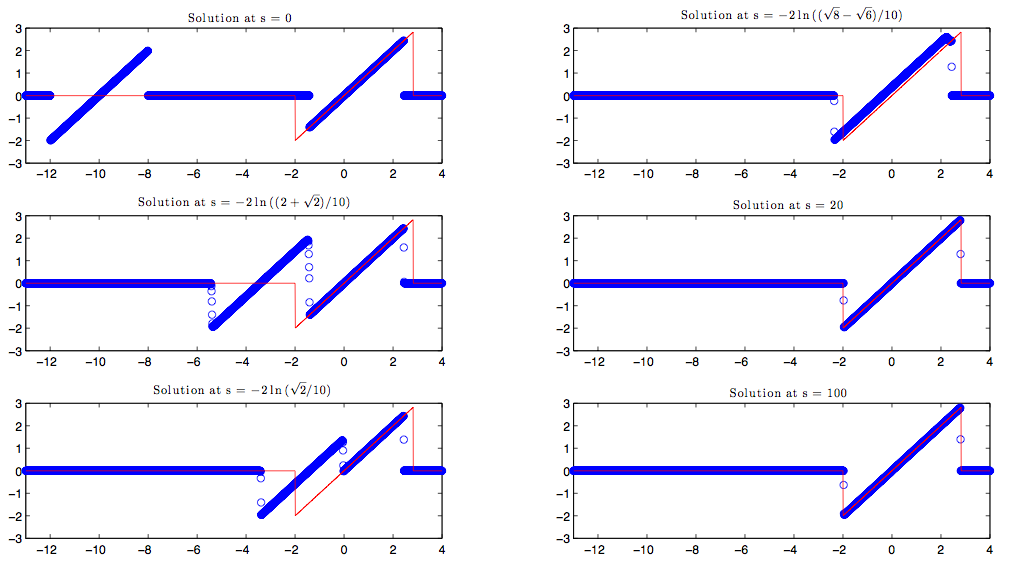}
	\caption{Convergence of the numerical solution of \eqref{eq.burgers.sta} using Engquist-Osher scheme (circle dots) to the asymptotic N-wave (solid line). We take $\Delta \xi=0.01$ and $\Delta s=0.0005$.}
	\label{fig:sta_sol2}
\end{figure}

Meanwhile, the artificial viscosity of the Lax-Friedrichs scheme starts becoming dominant, making the solution evolve to the parabolic diffusion wave, which is the steady state of the viscous version of \eqref{eq.burgers.sta} (cf. \cite{Kim2001}).

\begin{figure}[t!]
	\includegraphics[width=0.85\linewidth]{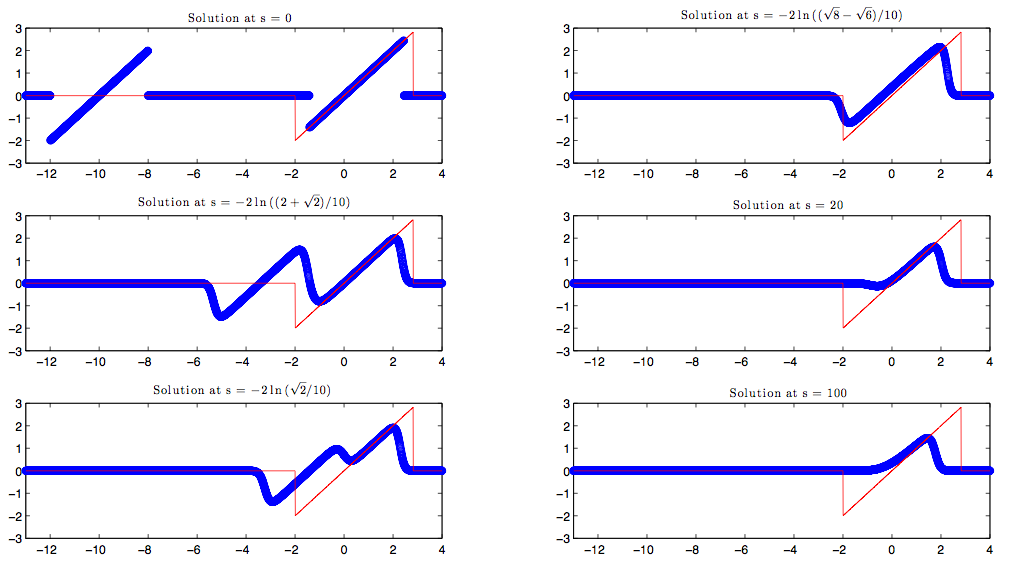}
	\caption{Numerical solution of \eqref{eq.burgers.sta} using the Lax-Friedrichs scheme (circle dots), taking $\Delta \xi=0.01$ and $\Delta s=0.0005$. The N-wave (solid line) is not reached, as it converges to the diffusion wave.}
	\label{fig:sta_sol3}
\end{figure}

\subsection{Computational benefits}
It is important to emphasize the benefits of using similarity variables to perform long-time simulations. For instance, let us consider the following initial data:
\begin{equation*}
	u_0(x)=\begin{cases} 2, & 0< x\le 2, \\ -1, & -1\le x\le 0, \\ 0, & \mbox{elsewhere.} \end{cases}
\end{equation*}
We  compute the numerical approximation of the corresponding solution to \eqref{eq.burgers} in two different ways, either in the original or in the self-similar variables,  and compare them to the exact solution, which can be computed explicitly. First we consider scheme \eqref{eq:consscheme}, based on physical variables, and then, scheme \eqref{sta.scheme} using similarity variables. In both cases we choose the Engquist-Osher numerical flux, i.e., \eqref{eo.numflux} and \eqref{eo.sta} respectively. Besides, in the latter case we use piecewise constant interpolation to recover the solution in the physical space so that we can compare it to the exact one.

In Table \ref{table.t100} we compare the errors of the solutions at $t=100$. Let us remark that, in order to avoid interferences of boundary conditions, we need to choose large enough spatial domains. We have chosen $[-20,30]$ for the case of physical variables and $[-3,4]$ for the other one. We consider $\dx=0.1$ and $\dx=0.01$ and $\dt=\dx/2$, accordingly to the CFL condition. Parameter $\dxi$ is chosen such that the $\|\cdot\|_{1,\Delta}$ error made is similar to the corresponding case, while we take $\ds=\dxi/20$. We make the same comparison in Table \ref{table.t1000} for the solutions at $t=1000$. In this case, we have taken $[-50,100]$ and $[-3,4]$ space intervals, respectively. The criteria for the mesh-size and time-step are the same as above.

We observe that to obtain similar accuracy, we need much less nodes and time iterations to compute the numerical solution at a given time. In fact, the results using similarity variables could be improved using a higher order reconstruction of the solution, instead of piecewise constant, when doing the change of variables to recover the physical solution.

\renewcommand*\arraystretch{1.2}
\begin{table}[t!]
	\begin{tabular}{l||c|c||c|c|c|}
		\cline{2-6}
		\multicolumn{1}{l|}{} & Nodes & Time-steps & $\|u_\Delta-u\|_{1,\Delta}$ & $\|u_\Delta-u\|_{2,\Delta}$ & $\|u_\Delta-u\|_{\infty,\Delta}$ \\
		\cline{2-6}
		\multicolumn{1}{l}{} \\ [-13pt]
		\hline
		\multicolumn{1}{|l||}{Physical variables} & 501 & 2001 & 0.2140 & 0.1352 & 0.2745 \\
		\hline
		\multicolumn{1}{l}{} \\ [-13pt]
		\hline
		\multicolumn{1}{|l||}{Similarity variables} & 100 & 1306 & 0.2057 & 0.1136 & 0.2543 \\
		\hline
	\end{tabular}
	\vskip 10pt
	\begin{tabular}{l||c|c||c|c|c|}
		\cline{2-6}
		\multicolumn{1}{l|}{} & Nodes & Time-steps & $\|u_\Delta-u\|_{1,\Delta}$ & $\|u_\Delta-u\|_{2,\Delta}$ & $\|u_\Delta-u\|_{\infty,\Delta}$ \\
		\cline{2-6}
		\multicolumn{1}{l}{} \\ [-13pt]
		\hline
		\multicolumn{1}{|l||}{Physical variables} & 5001 & 20001 & 0.0280 & 0.0517 & 0.2828 \\
		\hline
		\multicolumn{1}{l}{} \\ [-13pt]
		\hline
		\multicolumn{1}{|l||}{Similarity variables} & 750 & 9877 & 0.0276 & 0.0379 & 0.2465 \\
		\hline
		\multicolumn{1}{l}{} \\
	\end{tabular}
	\caption{Comparison of solutions at $t=100$. We take $\dx=0.1$ (top) and $\dx=0.01$ (bottom). We choose $\dxi$ such that the $\|\cdot\|_{1,\Delta}$ error is similar. The time-steps are $\dt=\dx/2$ and $\ds=\dxi/20$, respectively, enough to satisfy the CFL condition.}
	\label{table.t100}
\end{table}

\renewcommand*\arraystretch{1.2}
\begin{table}[t!]
	\begin{tabular}{l||c|c||c|c|c|}
		\cline{2-6}
		\multicolumn{1}{l|}{} & Nodes & Time-steps & $\|u_\Delta-u\|_{1,\Delta}$ & $\|u_\Delta-u\|_{2,\Delta}$ & $\|u_\Delta-u\|_{\infty,\Delta}$ \\
		\cline{2-6}
		\multicolumn{1}{l}{} \\ [-13pt]
		\hline
		\multicolumn{1}{|l||}{Physical variables} & 1501 & 19987 & 0.0867 & 0.0482 & 0.0893 \\
		\hline
		\multicolumn{1}{l}{} \\ [-13pt]
		\hline
		\multicolumn{1}{|l||}{Similarity variables} & 215 & 4225 & 0.0897 & 0.0332 & 0.0367 \\
		\hline
		\multicolumn{1}{l}{} \\
	\end{tabular}
	\vskip 10pt
	\begin{tabular}{l||c|c||c|c|c|}
		\cline{2-6}
		\multicolumn{1}{l|}{} & Nodes & Time-steps & $\|u_\Delta-u\|_{1,\Delta}$ & $\|u_\Delta-u\|_{2,\Delta}$ & $\|u_\Delta-u\|_{\infty,\Delta}$ \\
		\cline{2-6}
		\multicolumn{1}{l}{} \\ [-13pt]
		\hline
		\multicolumn{1}{|l||}{Physical variables} & 15001 & 199867 & 0.0093 & 0.0118 & 0.0816 \\
		\hline
		\multicolumn{1}{l}{} \\ [-13pt]
		\hline
		\multicolumn{1}{|l||}{Similarity variables} & 2000 & 39459 & 0.0094 & 0.0106 & 0.0233 \\
		\hline
		\multicolumn{1}{l}{} \\
	\end{tabular}
	\caption{Comparison of solutions at $t=1000$. We take $\dx=0.1$ (top) and $\dx=0.01$ (bottom). We choose $\dxi$ such that the $\|\cdot\|_{1,\Delta}$ error is similar. The time-steps are $\dt=\dx/2$ and $\ds=\dxi/20$, respectively, enough to satisfy the CFL condition.}
	\label{table.t1000}
\end{table}


\section{Generalizations and further comments}\label{section6}
Until now we have considered only the Burgers equation and Lax-Friedrichs, Engquist-Osher and Godunov schemes, but the described techniques can be extended to more general types of fluxes and numerical schemes. 

Regarding the latter, the main requirement is that the chosen scheme must verify the OSLC, so that the decay estimates can be used to guarantee the compactness of the rescaled solutions.  The homogeneity of the dissipation, as defined at the end of Section \ref{section2}, will indicate if the introduced artificial viscosity is strong enough to modify the asymptotic behavior of the numerical solution or if it preserves the continuous property.

As for the fluxes, it is worth to say that the N-wave appearing as the asymptotic profile is a common characteristic for all 1D scalar conservation laws with uniformly convex flux, that is, for those with $f''(u)\ge \gamma >0$, with $\gamma>0$. For that reason, one expects to observe the same phenomena in the discrete level as the ones described in this paper. As we said before, the OSLC will play a key role. Nevertheless, in this more general situation, obtaining the homogeneity $\alpha$ of the dissipation might be not so straightforward. The coefficient $R$ of the Lax-Friedrichs scheme does not depend on $f$, so it will not develop the N-wave regardless the flux we consider. The analysis of Engquist-Osher and Godunov schemes is more delicate, since their coefficients of viscosity are strictly related to the flux. In any case, whenever $\alpha\ge1$, the asymptotic profile will be the desired N-wave. The conclusion is the same for any uniformly concave flux, just by considering the reflected N-wave.

The analysis is also valid for some type of odd fluxes, those that are concave on one side of their axis of symmetry and convex in the other. Nevertheless, there will be no difference in the asymptotic profile anymore. For instance, let us consider equation \eqref{eq:hyp} with flux $f(u)=|u|u/2$. Then, the asymptotic behavior is the one stated by the following theorem.

\begin{theorem}\label{asymptoticodd}
Let $u_0\in L^1(\R)$ and choose mesh-size parameters $\dx$ and $\dt$ satisfying the CFL condition $\lambda\|u^n\|_{\infty,\Delta}\le1$, $\lambda=\dt/\dx$. Let $u_\Delta$ be the corresponding solution of the discrete scheme \eqref{num} for the hyperbolic conservation law \eqref{eq:hyp} with flux $f(u)=|u|u/2$. Then, for any $p\in [1,\infty)$, the following holds 
\begin{equation}\label{odd.limit.infty}
	\lim _{t\rightarrow \infty} t^{\frac 12(1-\frac 1p)}\|u_\Delta(t)-w(t)\|_{L^p(\rr)}=0,
\end{equation}
where the profile $w$ is as follows:
\begin{enumerate}
	\item for the Lax-Friedrichs scheme, $w=w_{M_\Delta}$ defined in \eqref{limit1}.
	\item for Engquist-Osher and Godunov schemes, $w=w_{0,M_\Delta}$ if $M_\Delta>0$ or $w=w_{M_\Delta,0}$ if $M_\Delta<0$, both given by \eqref{limit2}.
\end{enumerate}
\end{theorem}

The proof of this result is analogous to the one of Theorem \ref{asymptotic}. In this case, the key point now is the uniqueness of solution (see  \cite{Liu1984}) of the equation
\begin{equation*}
	\begin{cases}
		\displaystyle u^\infty_{t}+ \left(\frac{u^\infty|u^\infty|}{2}\right)_x=\frac{\Delta x^2}{2\Delta t} u^\infty_{xx},\\[10pt]
		u(0)=M\delta_0,
	\end{cases}	
\end{equation*}
which is the one appearing for Lax-Friedrichs, and of the equation
\begin{equation*}
	\begin{cases}
		\displaystyle u^\infty_{t}+ \left(\frac{u^\infty|u^\infty|}{2}\right)_x=0,\\[10pt]
		u(0)=M\delta_0,
	\end{cases}
\end{equation*}
corresponding to Engquist-Osher and Godunov. Moreover, both converge to the same continuous N-wave, which has unique sign, when $\dx,\dt\to0$.

\section*{Aknowledgements}
The authors would like to thank you J. J. Alonso and F. Palacios (Stanford) for stimulation discussions on sonic-boom minimization that led to the questions addressed in this paper.

This work is supported by the Grant MTM2011-29306-C02-00 of the MICINN (Spain), the Advanced Grant FP7-246775 of the European Research Council Executive Agency and the Grant PI2010-04 of the Basque Government. A.~Pozo is supported also by the Grant BFI-2010-339. L.~Ignat was also partially supported by Grant PN-II-ID-PCE-2012-4-0021 of the Romanian National Authority for Scientific Research, CNCS-UEFISCDI.

\bibliographystyle{amsplain}
\bibliography{library}

\end{document}